\documentclass[11pt]{amsart}

\usepackage{amssymb,amsfonts}
\usepackage[all,arc]{xy}
\usepackage{enumerate}
\usepackage{mathrsfs}
\usepackage{eufrak}
\usepackage{setspace}
\usepackage{color}
\usepackage{graphicx}
\usepackage{pdfpages}
\usepackage{enumitem}

\newtheorem{theorem}{Theorem}[section]
\newtheorem{proposition}[theorem]{Proposition}

\newtheorem{corollary}[theorem]{Corollary}
\newtheorem{definition}[theorem]{Definition}

\newtheorem{example}[theorem]{Example}

\theoremstyle{remark}
\newtheorem{remark}[theorem]{Remark}

\renewenvironment{proof}{{\noindent\bf Proof.}}{\hfill $\Box$\par\vskip3mm}

\newcommand{\Cc}{\mathcal{C}}

\theoremstyle{definition}

\theoremstyle{remark}

\makeatletter
\let\c@equation\c@thm
\makeatother
\numberwithin{equation}{section}

\title{ Cofree Objects in The Centralizer and The Center Categories }
\author{ADNAN H. ABDULWAHID}
\date{}

\begin{document}

\begin{abstract}
\noindent We study cocompleteness, co-wellpoweredness  and generators in the centralizer category of an object or morphism in a monoidal  category, and the center or the weak center of a monoidal category. We explicitly give some answers for when colimits, cocompleteness, co-wellpoweredness and generators in these monoidal categories can be inherited from their base monidal categories. Most importantly, we  investigate  cofree objects of comonoids in these monoidal categories.   
\end{abstract}

\date{}

\maketitle


\noindent

\section{\textbf{Introduction}}
Universal properties are significantly considered as one of the most important concepts in mathematics. In deed, they can be thought of as the skeleton of all mathematics concepts. They show how the objects and the morphisms being unified and described compatibly relate the whole category that they live in. Lots of influential concepts, such as kernels, cokernels, products, coproducts, limits, colimits, etc, are essentially involved with universal properties. Perhaps the most important notion concerned with them is the concept of adjoint functors. It is simply because ``Adjoint functors arise everywhere"  \cite[p. (vii)]{Mac Lane1}. Significantly, free and cofree objects play a crucial role in recasting the adjunctions of the forgetful functors in terms of comma categories. For fundamentals concepts and examples of adjoint functors, we refer the reader to \cite{Leinster1}, \cite{Mac Lane1}, \cite{Awodey}, \cite{Rotman}, \cite{Pareigis}, or \cite{Mitchell}. For the basic notions of  comma categories, we refer to  \cite{Leinster1} and \cite{Mac Lane1}. \\
  
Let $\mathfrak{X}$ be a category. A \textit{concrete category} over $\mathfrak{X}$ is a pair $(\mathfrak{A},\mathfrak{U})$, where $\mathfrak{A}$ is a category and $\mathfrak{U}: \mathfrak{A} \rightarrow \mathfrak{X}$ is a faithful functor \cite[p. 61]{Adamek}. Let $(\mathfrak{A},\mathfrak{U})$ be a concrete category over $\mathfrak{X}$. Following \cite[p. 140-143]{Adamek}, a \textit{free object} over $\mathfrak{X}$-object $X$ is an $\mathfrak{A}$-object $A$ such that there exists a \textit{universal arrow}  $(A,u)$ over $X$; that is, $u:X \rightarrow \mathfrak{U}A$ such that for every arrow $f:X \rightarrow \mathfrak{U}B$, there exists a unique morphism $f':A \rightarrow B$ in $\mathfrak{A}$ such that $\mathfrak{U}f' u=f$. We also say that $(A,u)$ is the free object over $X$.  A concrete category  $(\mathfrak{A},\mathfrak{U})$ over $\mathfrak{X}$ is said to \textit{have free objects} provided that for each $\mathfrak{X}$-object $X$, there exists a universal arrow over $X$. For example, the category $Vect_{\mathbb{K}}$ of vector spaces over a field $\mathbb{K}$ has free objects. So do the category \textbf{Top} of topological spaces and the category of \textbf{Grp} of groups. However, some interesting categories do not have free objects \cite[p. 142]{Adamek}). \\

Dually, \textit{co-universal arrows}, \textit{cofree objects} and categories that \textit{have cofree objects} can be defined. For the basic concepts of concrete categories, free objects, and cofree objects, we refer the reader to \cite[p. 138-155]{Kilp}. \\

It turns out that a  concrete $(\mathfrak{A},\mathfrak{U})$ over $\mathfrak{X}$ has (co)free objects if and only if the functor that builds up (co)free object is a (right) left adjoint to the faithful functor $\mathfrak{U}: \mathfrak{A} \rightarrow \mathfrak{X}$.\\

Although cofree objects are the dual of free objects, the behavior of cofree objects is more complicated than the one of free objects. Furthermore, studying such behavior cannot be obtained by studying free objects because ``the categories considered are not selfdual generally"  \cite[p. 149]{Kilp}. In this paper, we are interested in investigating cofree objects in the centralizer category of an object or morphism in a monoidal  category and the center or the weak center  of a monoidal category. For the basic notions of monoidal categories, we refer the reader to \cite{Etingof}, \cite{Bakalov}, and \cite[Chapter 6]{Borceux2}.\\

More recently, these monoidal categories play a vibrant role in characterizing and identifying many of interesting categories. For instance, to show that two finite tensor categories are Morita equivalent, it suffices to show that their centers are equivalent as braided tensor categories \cite[p. 222]{Etingof}. Another example is to show that a fusion category is group-theoretical, it is sufficient to show its center contains a Lagrangian subcategory \cite[p. 313]{Etingof}. \\
In addition, there is a special importance for the center of a finite tensor category in finding its Frobenius-Perron dimension. This comes from the fact that for any finite tensor category $\mathscr{C}$, we have  $FPdim(\mathcal{Z}(\mathscr{C})) = FPdim(\mathscr{C})^2$ \cite[p. 168]{Etingof}. We refer to \cite{Joyal2}  for basics on centralizer categories while we refer to   \cite[p. 76]{Street} and \cite[p. 162]{Etingof} for basics on center categories. \\

Explicitly, the problem can be formulated as follows. Let $\Cc$ be a  monoidal category.  Fix an object $X$ and a morphism $ h:A \rightarrow B$ in $\Cc$. For any  $\mathscr{A} \in \{ \mathcal{Z}_h(\Cc), \mathcal{Z}_X(\Cc), \mathcal{Z}(\Cc) ,  \mathcal{Z}_{\omega}(\Cc) \}$, let  $\mathscr{U}_{\!_{{\mathscr{A}}}}: CoMon(\mathscr{A}) \rightarrow \mathscr{A}$  be the forgetful functor corresponding to  $\mathscr{A}$. Does  $\mathscr{U}_{\!_{{\mathscr{A}}}}$ have a right adjoint?   \\

A reasonably expected machinery for the answer of this question is  
the dual of Special Adjoint Functor Theorem (D-SAFT). \\

We start our inspection by studying the cocompleteness in  $\mathscr{A}$, and we give some answers for the question: under what conditions the colimits of objects in $\mathscr{A}$ can be obtained from the corresponding construction for objects in $\Cc$. The later implicitly implies that the forgetful functor  $\mathscr{U}_{\!_{{\mathscr{A}}}}$ is cocontinuous. \\

Next, we study some conditions that make the co-wellpoweredness of the  category $\mathscr{A}$ can be inherited from $\Cc$.\\
We also show how the braiding forces the category $\mathscr{A}$ to inherit generators from its base category $\Cc$.\\

Finally, we apply the mechanism  of D-SAFT for each case. Furthermore, we try to visualize some interesting consequences by studying the braid category.

\section{\textbf{Preliminaries}}\label{s.p}

From now on let $(\Cc,\otimes,I)$ be a monoidal category and for every $X \in \Cc$, $\mathcal{P}_X, \mathcal{Q}_X$  the functors defined by  define the functors

\begin{center}
•$\mathcal{P}_X = X \otimes -: \Cc \rightarrow \Cc$, $M \mapsto X \otimes M$, \\
$\mathcal{Q}_X = - \otimes X: \Cc \rightarrow \Cc$, $M \mapsto M \otimes X$.
\end{center}

\begin{theorem} \cite[p. 148]{Freyd} \label{p.SAFT} 
If $\mathfrak{A}$ is cocomplete, co-wellpowered and with a generating set, then every cocontinuous functor from $\mathfrak{A}$ to a locally small category has a right adjoint.
\end{theorem}

\begin{definition} \label{def.1} \cite[p. 284]{Schauenburg} 
The \textit{(left) weak center} of $\Cc$ denoted $\mathcal{Z}_{\omega}(\Cc)$ is a category whose objects are pairs $(X,\sigma)$ in which $X \in \Cc$ and $\sigma_{\!_{Y}}: X \otimes Y \rightarrow Y \otimes X$ natural in $Y \in \Cc$ and satisfies $\sigma_{\!_{Y\otimes Z}} = (id_{\!_{Y}} \otimes \sigma_{\!_{Z}})(\sigma_{\!_{Y}} \otimes id_{\!_{X}})$ for all $Y,Z \in \Cc$ and $\sigma_{\!_{X,I}} = id_{\!_{X}}$. 

An arrow $f : (A,\sigma) \rightarrow (B, \tau)$ in $\mathcal{Z}_{\omega}(\Cc)$ is an arrow $f : A \rightarrow B$ in $\Cc$ such that, for all $X \in \Cc$, the following diagram

\begin{equation} \label{diag.eq1}
\xymatrix{
A \otimes X \ar[d]_{\sigma_{\!_{X}}} \ar[rr]^-{(f \otimes id_{\!_{X}})}
& & B \otimes X \ar[d]^{ \tau_{\!_{X}}} \\
X \otimes A \ar[rr]_-{(id_{\!_{X}} \otimes f)} & & X \otimes B }
\end{equation} commutes.\\

The category  $\mathcal{Z}_{\omega}(\Cc)$ is monoidal with 
\begin{equation} \label{def.eq2}
(A,\sigma) \otimes (B,\tau) = (A \otimes B,\delta),
\end{equation}  where 
\begin{equation} \label{def.eq3}
\delta_{\!_{X}}:  A \otimes B \otimes X \rightarrow X \otimes  A \otimes B  = (\sigma_{\!_{X}}\otimes id_{\!_{B}} )(id_{\!_{A}} \otimes \tau_{\!_{X}}).
\end{equation}
\end{definition}

\begin{definition} \label{def.2} \cite[p. 76]{Street} 
The \textit{center} of $\Cc$ denoted $\mathcal{Z}(\Cc)$ is a category whose objects are pairs $(A,\sigma)$ where $A \in \Cc$ and $\sigma: A \otimes - \xrightarrow{ \sim } - \otimes A$ is a natural isomorphism such that the following conditions hold:
 \begin{equation} \label{def.eq4}
\sigma_{\!_{I}} = id_{\!_{A}}
\end{equation} (more precisely, $\sigma_{\!_{I}}$ is the composite of the canonical isomorphisms $A \otimes I \cong A \cong I \otimes A$), and
\begin{equation} \label{def.eq5} 
\sigma_{\!_{X\otimes Y}} = (id_{\!_{X}} \otimes \sigma_{\!_{Y}})(\sigma_{\!_{X}} \otimes id_{\!_{Y}}) 
\end{equation} for all $X,Y \in \Cc$.

An arrow $f : (A,\sigma) \rightarrow (B,\tau)$ in $\mathcal{Z}(\Cc)$ is an arrow $f : A \rightarrow B$ in $\Cc$ such that, for all $X \in \Cc$, the following diagram is commutative
\begin{equation} \label{diag.eq6}
\xymatrix{
A \otimes X \ar[d]_{\sigma_{\!_{X}}}^{\sim} \ar[rr]^-{(f \otimes id_{\!_{X}})}
&& B \otimes X \ar[d]^{ \tau_{\!_{X}}}_{\sim} \\
X \otimes A \ar[rr]_-{(id_{\!_{X}} \otimes f)} && X \otimes B }
\end{equation}
The category  $\mathcal{Z}(\Cc)$ is monoidal with:
\begin{equation} \label{def.eq7}
(A,\sigma) \otimes (B,\tau) = (A \otimes B,\delta),
\end{equation}  where 
\begin{equation} \label{def.eq8}
\delta_{\!_{X}} = (\sigma_{\!_{X}}\otimes id_{\!_{B}} )(id_{\!_{A}} \otimes \tau_{\!_{X}}).
\end{equation}
The category  $\mathcal{Z}(\Cc)$ is braided via
\begin{equation}  \label{def.eq9}
\Psi_{\!_{(A,\sigma), (B,\tau)}} = \sigma_{\!_{B}}:(A,\sigma) \otimes (B,\tau) \rightarrow  (B,\tau) \otimes (A,\sigma).
\end{equation}

 \end{definition}

\begin{remark} \label{rem.1} \cite[p. 76]{Street} 
The condition \ref{def.eq4} on objects of $\mathcal{Z}(\Cc)$ is redundant.
\end{remark}

\begin{definition} \label{def.3} \cite[p. 46-47]{Joyal2}
The \textit{centralizer}  $\mathcal{Z}_X(\Cc)$ of an object $X \in \Cc$ is the category whose 
objects are pairs $(A,\alpha)$, where $A \in \Cc$ and $\alpha: A \otimes X \xrightarrow{ \sim } X \otimes A$. 

An arrow $f : (A,\alpha) \rightarrow (B, \beta)$ in $\mathcal{Z}_X(\Cc)$ is an arrow $f: A \rightarrow B$ in $\Cc$ such that the following diagram is commutative
\begin{equation} \label{diag.eq10}
\xymatrix{
A \otimes X \ar[d]_{\alpha}^{\sim} \ar[rr]^-{(f \otimes id_{\!_{X}})}
&& B \otimes X \ar[d]^{ \beta}_{\sim} \\
X \otimes A \ar[rr]_-{(id_{\!_{X}} \otimes f)} && X \otimes B }
\end{equation}
This becomes a monoidal category with 
\begin{equation} \label{def.eq11}
(A,\sigma) \otimes (B,\tau) = (A \otimes B,\gamma),
\end{equation}  where 
\begin{equation} \label{def.eq12}
\gamma = (\alpha\otimes id_{\!_{Y}} )(id_{\!_{X}} \otimes \beta_{\!_{X}}).
\end{equation}
 \end{definition}

\begin{definition} \label{def.4}  \cite[p. 49]{Joyal2} 
The \textit{centralizer}  $\mathcal{Z}_h(\Cc)$ of an arrow $h: A \rightarrow B$ in $\Cc$ is the category whose objects are triples $(X,\alpha,\beta)$, where $X \in \Cc$ and $\alpha: A \otimes X \xrightarrow{ \sim } X \otimes A, \beta: B \otimes X \xrightarrow{ \sim } X \otimes B$ are isomorphisms such that the following diagram is commutative
\begin{equation} \label{diag.eq13}
\xymatrix{
A \otimes X \ar[d]_{\alpha}^{\sim} \ar[rr]^-{(h \otimes id_{\!_{X}})}
&& B \otimes X \ar[d]^{ \beta}_{\sim} \\
X \otimes A \ar[rr]_-{(id_{\!_{X}} \otimes h)} && X \otimes B }
\end{equation}

An arrow $f : (X,\alpha,\beta) \rightarrow (Y,\alpha',\beta')$ in $\mathcal{Z}_h(\Cc)$ is an arrow $f: X \rightarrow Y$ in $\Cc$ such that the following diagrams are commutative

\begin{equation} \label{diag.eq14}
\xymatrix{
A \otimes X \ar[d]_-{(id_{\!_{A}} \otimes f)} \ar[rr]^{\alpha}_{\sim}
&& X \otimes A \ar[d]^-{ (f \otimes id_{\!_{A}})} \\
A \otimes Y \ar[rr]_{\alpha'}^{\sim} && Y \otimes A }
\hspace{35pt} 
\xymatrix{
B \otimes X \ar[d]_-{(id_{\!_{B}} \otimes f)} \ar[rr]^{\beta}_{\sim}
&& X \otimes B \ar[d]^-{ (f \otimes id_{\!_{B}})} \\
B \otimes Y \ar[rr]_{\beta'}^{\sim} && Y \otimes B }
\end{equation}
\end{definition}

\begin{remark} \label{rem.2}
The category $\mathcal{Z}_h(\Cc)$ was introduced in  \cite[p. 49]{Joyal2} as an essential part of the proof of lemma 7, and the authors implicitly indicated that it is a monoidal category. For  convenience, we explicitly show that  the category $\mathcal{Z}_h(\Cc)$ is monoidal.
\end{remark}

\begin{proposition} \label{p.1} 
Let $h: A \rightarrow B$ be an arrow in $\Cc$. Then the category $\mathcal{Z}_h(\Cc)$ is monoidal with 
\begin{equation} \label{p.eq15}
(X,\alpha,\beta) \otimes (Y,\alpha',\beta') = (X \otimes Y,\bar{\alpha},\bar{\beta}),
\end{equation} where $\bar{\alpha},\bar{\beta}$ are given respectively by the following compositions

\begin{equation} \label{p.eq16}
A \otimes X \otimes Y \xrightarrow[\sim]{\alpha \otimes id_{\!_{Y}}} X \otimes A \otimes Y \xrightarrow[\sim]{id_{\!_{X}} \otimes \alpha'} X \otimes Y \otimes A 
\end{equation} 
\begin{equation} \label{p.eq17}
B \otimes X \otimes Y \xrightarrow[\sim]{\beta \otimes id_{\!_{Y}}} X \otimes B \otimes Y \xrightarrow[\sim]{id_{\!_{X}} \otimes \beta'} X \otimes Y \otimes B 
\end{equation}\\
\end{proposition}

\begin{proof}

\vspace{.2cm}

\begin{tabular}{lllll}
$\bar{\beta} (h \otimes id_{\!_{X\otimes Y}}) $  &  $=(id_{\!_{X}} \otimes \beta') (\beta \otimes id_{\!_{Y}})(h \otimes id_{\!_{X\otimes Y}})$\\
&  (definition of $\bar{\beta}$)\\
 & $=(id_{\!_{X}} \otimes \beta') (\beta \otimes id_{\!_{Y}})(h \otimes id_{\!_{X}} \otimes id_{\!_{Y}})$\\
&  ($\otimes$ is a bifunctor)\\
 & $=(id_{\!_{X}} \otimes \beta') (\beta (h \otimes id_{\!_{X}})\otimes id_{\!_{Y}})$\\
 &   (naturality of $\otimes$) \\
  & $=(id_{\!_{X}} \otimes \beta') ((id_{\!_{X}} \otimes h) \alpha \otimes id_{\!_{Y}})$\\
 &   (since $X \in \mathcal{Z}_h(\Cc)$) \\
 & $=(id_{\!_{X}} \otimes \beta') (id_{\!_{X}} \otimes h \otimes id_{\!_{Y}}) (\alpha \otimes id_{\!_{Y}})$\\
 &  (naturality of $\otimes$) \\
 & $=(id_{\!_{X}} \otimes \beta'(h \otimes id_{\!_{Y}})) (\alpha \otimes id_{\!_{Y}})$\\
 &  (naturality of $\otimes$) \\
& $=(id_{\!_{X}} \otimes (id_{\!_{Y}}\otimes h) \alpha') (\alpha \otimes id_{\!_{Y}})$\\
 &    (since $Y \in \mathcal{Z}_h(\Cc)$) \\
 & $=(id_{\!_{X}} \otimes id_{\!_{Y}} \otimes h) (id_{\!_{X}} \otimes\alpha') (\alpha \otimes id_{\!_{Y}})$\\
 &   (naturality of $\otimes$)  \\
 & $=(id_{\!_{X\otimes Y}} \otimes h) (id_{\!_{X}} \otimes\alpha') (\alpha \otimes id_{\!_{Y}})$\\
 &   ($\otimes$ is a bifunctor)  \\
  & $=(id_{\!_{X\otimes Y}} \otimes h) \bar{\alpha} $\\
 &   (definition of $\bar{\alpha}$)  \\
     
\end{tabular}\\
  
\vspace{.3cm}
  
Therefore, $X \otimes Y \in \mathcal{Z}_h(\Cc)$ and thus the category $\mathcal{Z}_h(\Cc)$ is monoidal.

\end{proof}

\vspace{.2cm}
 
\begin{remark} \textbf{•} \label{r.Z.X.Z.h}  
\begin{enumerate}[label=(\roman*)]
\item For all $X \in \Cc$, we have the following evaluation functor 
\begin{equation} \label{diag.embedding}
 \xymatrix{
\mathcal{Z}(\Cc) \ar[r]^-{\mathscr{H}_{\!_{X}}} & \mathcal{Z}_X(\Cc) 
} 
\end{equation} 
where $\mathscr{H}_{\!_{X}}$ defined by $(A,\sigma) \mapsto (A,\sigma_{\!_{X}})$. It turns out that if $(A,\sigma) \in \mathcal{Z}(\Cc)$, then $(A,\sigma_{\!_{X}}) \in \mathcal{Z}_X(\Cc)$, for every $X \in \Cc$. However, to show that an object $(A,\sigma) \in \mathcal{Z}(\Cc)$, by Remark \ref{rem.1}, it suffices to show that $(A,\sigma_{\!_{X}}) \in \mathcal{Z}_X(\Cc)$ for every $X \in \Cc$,  $\sigma$ is natural and the condition \ref{def.eq5} holds as well. 

\item For all $A \in \Cc$, if  $ (X,\alpha,\beta) \in \mathcal{Z}_{id_{\!_{A}}}(\Cc)$, then Definition \ref {def.4} implies  that $\alpha = \beta$. This gives rise to an isomorphism given by

\begin{equation} \label{diag.Z.X.Z.h}
 \xymatrix{
\mathcal{Z}_A(\Cc) \ar@/^/[rr]^-{\mathscr{S}_{\!_{A}}} & & \mathcal{Z}_{id_{\!_{A}}}(\Cc) \ar@/^/[ll]^-{\mathscr{T}_{\!_{A}}}
} 
\end{equation} 

where $\mathscr{S}_{\!_{A}}$ is defined by $(X,\alpha, \alpha) \mapsto (X,\alpha^{-1})$, and  $\mathscr{T}_{\!_{A}}$ is defined by $(X,\alpha) \mapsto (X,\alpha^{-1}, \alpha^{-1})$.\\
Thus, we have $\mathcal{Z}_A(\Cc) \cong  \mathcal{Z}_{id_{\!_{A}}}(\Cc)$.  It turns out that the centralizer category of an object $A$ in $\Cc$ can be identified as the centralizer category of the identity morphism of $A$ in $\Cc$. However, we will explicitly study the centralizer category of an object due to the discussion of part (i). 

\item From Definitions \ref{def.2}, \ref{def.3}, there is an embedding 
$ \mathcal{Z}(\Cc) \hookrightarrow \mathcal{Z}_{\omega}(\Cc)$. 
\end{enumerate}

\end{remark}

\section{\textbf{Cocompleteness}}\label{s.cocomp}
Recall that a  category $\mathfrak{C}$ is cocomplete when every functor $\mathfrak{F}: \mathfrak{D} \rightarrow \mathfrak{C}$,
with $\mathfrak{D}$ a small category has a colimit \cite{Borceux1}. For the basic notions of cocomplete categories and examples, we refer to  \cite{Adamek}, \cite{Borceux1}, or \cite{Schubert}. A  functor is \textit{cocontinuous} if it preserves all small colimits \cite[p. 142]{Freyd}.

\begin{proposition}\label{p.cocomp-cowell} \cite[p. 5]{Abdulwahid} 
Let $CoMon(\Cc)$ be the category of comonoids of $\Cc$ and $\mathscr{U}:CoMon(\Cc)\rightarrow \Cc$ the forgetful functor. If $\Cc$ is cocomplete, then $CoMon(\Cc)$ is cocomplete and $U$ is cocontinuous. 
\end{proposition}
\begin{proof}

\end{proof}

\begin{proposition}\label{p.cocomp1} 
Let $\Cc$ be a cocomplete category and  $h: A \rightarrow B$  an arrow in $\Cc$. If $\mathcal{P}_J$ and $\mathcal{Q}_J$ are cocontinuous $\forall J \in \{ A,B\}$, then $\mathcal{Z}_h(\Cc)$ is cocomplete and the forgetful functor $\mathscr{U}: \mathcal{Z}_h(\Cc) \rightarrow \mathcal{C}$ is cocontinuous. Furthermore, the colimit of objects in $\mathcal{Z}_h(\Cc)$ can be obtained by the corresponding construction for objects in $\Cc$.
\end{proposition}
\begin{proof}
Let $\mathscr{D}$ be a small category, and let $\mathscr{F}:\mathscr{D} \rightarrow \mathcal{Z}_h(\Cc)$ be a functor. Since $\mathcal{C}$ is a cocomplete category, the functor $\mathscr{U} \mathscr{F}$ has a colimit $(C,(\phi_{\!_{D}})_{\!_{{D} \in \mathscr{D}}})$. Since  $\mathcal{P}_A$ is cocontinuous, $\mathcal{P}_A$ $ \mathscr{U} \mathscr{F}$ has a colimit $(\mathcal{P}_A (C),(\mathcal{P}_A (\phi_{\!_{D}}))_{\!_{{D} \in \mathscr{D}}})$. Equivalently, $(A \otimes C,(id_{\!_{A}} \otimes \phi_{\!_{D}})_{\!_{{D} \in \mathscr{D}}})$ is a colimit of $\mathcal{P}_A$ $ \mathscr{U} \mathscr{F}$. \\
First, we note that the functor $\mathscr{F}:\mathscr{D} \rightarrow \mathcal{Z}_h(\Cc)$ assigns to each object $D \in \mathscr{D}$ an object $(\mathscr{F}D,\alpha_{\!_{\mathscr{F}D}}, \beta_{\!_{\mathscr{F}D}}) \in \mathcal{Z}_h(\Cc)$. We also have $\mathscr{F}f: (\mathscr{F}D,\alpha_{\!_{\mathscr{F}D}}, \beta_{\!_{\mathscr{F}D}}) \rightarrow (\mathscr{F}D',\alpha_{\!_{\mathscr{F}D'}}, \beta_{\!_{\mathscr{F}D'}})$ is an arrow in $\mathcal{Z}_h(\Cc)$ for every arrow $f: D \rightarrow D'$ in $\mathscr{D}$. Thus, we have the following commutative diagrams: 
\begin{equation} \label{diag.eq18}
\xymatrix{
A \otimes D \ar[d]_-{(id_{\!_{A}} \otimes \mathscr{F}f)} \ar[rr]^{\alpha_{\!_{\mathscr{F}D}}}_{\sim}
&& D \otimes A \ar[d]^-{ (\mathscr{F}f \otimes id_{\!_{A}})} \\
A \otimes D' \ar[rr]_{\alpha_{\!_{\mathscr{F}D'}}}^{\sim} && D' \otimes A }
\hspace{35pt} 
\xymatrix{
B \otimes D \ar[d]_-{(id_{\!_{B}} \otimes \mathscr{F}f)} \ar[rr]^{\beta_{\!_{\mathscr{F}D}}}_{\sim}
&& D \otimes B \ar[d]^-{ (\mathscr{F}f \otimes id_{\!_{B}})} \\
B \otimes D' \ar[rr]_{\beta_{\!_{\mathscr{F}D'}}}^{\sim} &&D' \otimes B }
\end{equation} Now consider the following diagram
\begin{equation} \label{diag.eq19}
\xymatrix{
A \otimes \mathscr{F}D \ar[rrrr]^{id_{\!_{A}} \otimes \mathscr{F}f} \ar@/_/[dd]_{\alpha_{\!_{\mathscr{F}D}}}^{\sim} \ar@/_/[drr]_{id_{\!_{A}} \otimes \phi_{\!_{D}}} 
&&&& A \otimes \mathscr{F}D'  \ar@/^/[dll]^{id_{\!_{A}} \otimes \phi_{\!_{D'}}} \ar@/^/[dd]^{\alpha_{\!_{\mathscr{F}D'}}}_{\sim} \\
& & A \otimes C \ar@{.>}[d]^(.5){\exists ! \bar{\alpha}} \\
\mathscr{F}D \otimes A \ar@/_3pc/[rrrr]_{\mathscr{F}f \otimes id_{\!_{A}}} \ar@/_.3pc/[rr]_(.55){\phi_{\!_{D}} \otimes id_{\!_{A}}} & & C \otimes A && \mathscr{F}D' \otimes A \ar@/^.3pc/[ll]^(.55){\phi_{\!_{D'}} \otimes id_{\!_{A}}}
}
\end{equation}

For any $D$ in $\mathscr{D}$, we have  

 \vspace{.2cm}
 
\begin{tabular}{lllll}
$(\phi_{\!_{D'}} \otimes id_{\!_{A}})\alpha_{\!_{\mathscr{F}D'}} (id_{\!_{A}} \otimes \mathscr{F}f) $  &  $= (\phi_{\!_{D'}} \otimes id_{\!_{A}}) (\mathscr{F}f \otimes id_{\!_{A}}) \alpha_{\!_{\mathscr{F}D}} $\\
&  (by \ref{diag.eq18})\\
 & $= (\phi_{\!_{D'}} \mathscr{F}f \otimes id_{\!_{A}})  \alpha_{\!_{\mathscr{F}D}} $\\
&  (naturality of $\otimes$) \\
 & $= (\phi_{\!_{D}} \otimes id_{\!_{A}})  \alpha_{\!_{\mathscr{F}D}} $\\
&  (since $(C,(\phi_{\!_{D}})_{\!_{{D} \in \mathscr{D}}})$) is a cocone on $\mathscr{U} \mathscr{F}$) \\
  \end{tabular}\\
  
\vspace{.2cm}
  
Therefore, $(C \otimes A,((\phi_{\!_{D}} \otimes id_{\!_{A}})  \alpha_{\!_{\mathscr{F}D}})_{\!_{{D} \in \mathscr{D}}})$ is a cocone on $\mathcal{P}_A$ $ \mathscr{U} \mathscr{F}$. Since  $(A \otimes C,(id_{\!_{A}} \otimes \phi_{\!_{D}})_{\!_{{D} \in \mathscr{D}}})$ is a colimit of $\mathcal{P}_A$ $ \mathscr{U} \mathscr{F}$, there exits a unique arrow $\bar{\alpha}: A \otimes C \rightarrow C \otimes A$ in $\Cc$ with $\bar{\alpha} (id_{\!_{A}} \otimes \phi_{\!_{D}}) = (\phi_{\!_{D}} \otimes id_{\!_{A}}) \alpha_{\!_{\mathscr{F}D}} $.\\

Similarly, since  $\mathcal{Q}_A$ is cocontinuous, $\mathcal{Q}_A$ $ \mathscr{U} \mathscr{F}$ has a colimit of $(\mathcal{Q}_A (C),(\mathcal{Q}_A (\phi_{\!_{D}}))_{\!_{{D} \in \mathscr{D}}})$. So $(C\otimes A, (\phi_{\!_{D}} \otimes id_{\!_{A}})_{\!_{{D} \in \mathscr{D}}})$ is a colimit of $\mathcal{Q}_A$ $ \mathscr{U} \mathscr{F}$. \\

Correspondingly, we have 

\vspace{.3cm}

\begin{tabular}{lllll}
$(id_{\!_{A}}\otimes \phi_{\!_{D'}})\alpha_{\!_{\mathscr{F}D'}}^{-1} (\mathscr{F}f \otimes id_{\!_{A}}) $  &  $= (id_{\!_{A}}\otimes \phi_{\!_{D'}}) (id_{\!_{A}} \otimes \mathscr{F}f) \alpha_{\!_{\mathscr{F}D}}^{-1} $\\
&  (by \ref{diag.eq18})\\
 & $= (id_{\!_{A}}  \otimes \phi_{\!_{D'}} \mathscr{F}f)  \alpha_{\!_{\mathscr{F}D}}^{-1} $\\
&  (naturality of $\otimes$) \\
 & $= (id_{\!_{A}}  \otimes \phi_{\!_{D}}) \alpha_{\!_{\mathscr{F}D}}^{-1} $\\
&  (since $(C,(\phi_{\!_{D}})_{\!_{{D} \in \mathscr{D}}})$) is a cocone on $\mathscr{U} \mathscr{F}$) \\
  \end{tabular}\\
  
\vspace{.2cm}
for any $D$ in $\mathscr{D}$. Hence, $(A \otimes C,((id_{\!_{A}}  \otimes \phi_{\!_{D}}) \alpha_{\!_{\mathscr{F}D}}^{-1})_{\!_{{D} \in \mathscr{D}}})$ is a cocone on $\mathcal{P}_A$ $ \mathscr{U} \mathscr{F}$. Since  $(C \otimes A,( \phi_{\!_{D}} \otimes id_{\!_{A}} )_{\!_{{D} \in \mathscr{D}}})$ is a colimit of $\mathcal{Q}_A$ $ \mathscr{U} \mathscr{F}$, there exits a unique arrow $\bar{\alpha}': C \otimes A \rightarrow A \otimes C$ in $\Cc$ with $\bar{\alpha}' (\phi_{\!_{D}} \otimes id_{\!_{A}}) = (id_{\!_{A}} \otimes \phi_{\!_{D}}) \alpha_{\!_{\mathscr{F}D}}^{-1} $.\\

Therefore, we get the following commutative diagram
\begin{equation} \label{diag.eq20}
\xymatrix{
\mathscr{F}D \otimes A\ar[rrrr]^{\mathscr{F}f  \otimes id_{\!_{A}}} \ar@/_/[dd]_{\alpha_{\!_{\mathscr{F}D}}^{-1}}^{\sim} \ar@/_/[drr]_{\phi_{\!_{D}}\otimes id_{\!_{A}}} 
&&&& \mathscr{F}D' \otimes A \ar@/^/[dll]^{\phi_{\!_{D'}}\otimes id_{\!_{A}}} \ar@/^/[dd]^{\alpha_{\!_{\mathscr{F}D'}}^{-1}}_{\sim} \\
& & C \otimes A \ar@{.>}[d]^(.5){\exists ! \bar{\alpha}'} \\
A \otimes \mathscr{F}D \ar@/_3pc/[rrrr]_{id_{\!_{A}}\otimes \mathscr{F}f} \ar@/_.3pc/[rr]_(.55){id_{\!_{A}}\otimes \phi_{\!_{D}} } & & A \otimes C && A \otimes \mathscr{F}D' \ar@/^.3pc/[ll]^(.55){id_{\!_{A}}\otimes \phi_{\!_{D'}} }
}
\end{equation} 

Next, we show that $\bar{\alpha}$ is an invertible arrow. From the commutativity of the diagrams \ref{diag.eq19} and \ref{diag.eq20}, we have 
\begin{center}
•$\bar{\alpha} (id_{\!_{A}} \otimes \phi_{\!_{D}}) = (\phi_{\!_{D}} \otimes id_{\!_{A}}) \alpha_{\!_{\mathscr{F}D}} \Leftrightarrow \bar{\alpha} (id_{\!_{A}} \otimes \phi_{\!_{D}}) \alpha_{\!_{\mathscr{F}D}}^{-1} = (\phi_{\!_{D}} \otimes id_{\!_{A}}) \Leftrightarrow \bar{\alpha} \bar{\alpha}' (\phi_{\!_{D}} \otimes id_{\!_{A}}) = (\phi_{\!_{D}} \otimes id_{\!_{A}})$. \\
\end{center}

Obviously, $(C\otimes A, (\bar{\alpha} \bar{\alpha}' (\phi_{\!_{D}} \otimes id_{\!_{A}}))_{\!_{{D} \in \mathscr{D}}})$ is a cocone on $\mathcal{Q}_A$ $ \mathscr{U} \mathscr{F}$. Since $(C\otimes A, (\phi_{\!_{D}} \otimes id_{\!_{A}})_{\!_{{D} \in \mathscr{D}}})$ is a colimit of $\mathcal{Q}_A$ $ \mathscr{U} \mathscr{F}$, we have $\bar{\alpha} \bar{\alpha}' =   id_{\!_{C \otimes A}}$.\\

From the commutativity of the diagrams \ref{diag.eq19} and \ref{diag.eq20}, we have 
\begin{center}
•$ (id_{\!_{A}} \otimes \phi_{\!_{D}}) \alpha_{\!_{\mathscr{F}D}}^{-1} = \bar{\alpha}' (\phi_{\!_{D}} \otimes id_{\!_{A}}) \Leftrightarrow  (id_{\!_{A}} \otimes \phi_{\!_{D}}) = \bar{\alpha}' (\phi_{\!_{D}} \otimes id_{\!_{A}})  \alpha_{\!_{\mathscr{F}D}} \Leftrightarrow  (id_{\!_{A}} \otimes \phi_{\!_{D}}) = \bar{\alpha}' \bar{\alpha} (id_{\!_{A}} \otimes \phi_{\!_{D}})$. \\
\end{center}

Clearly, $(A\otimes C, (\bar{\alpha}' \bar{\alpha} (id_{\!_{A}} \otimes \phi_{\!_{D}}))_{\!_{{D} \in \mathscr{D}}})$ is a cocone on $\mathcal{P}_A$ $ \mathscr{U} \mathscr{F}$. Since $(A\otimes C, ( id_{\!_{A}} \otimes \phi_{\!_{D}})_{\!_{{D} \in \mathscr{D}}})$ is a colimit of $\mathcal{P}_A$ $ \mathscr{U} \mathscr{F}$, we have $\bar{\alpha}' \bar{\alpha} =   id_{\!_{A \otimes C}}$. Therefore, the arrow $\bar{\alpha}$ is invertible and  $\bar{\alpha}^{-1} =  \bar{\alpha}'$. \\

Replacing the object $A$ by $B$ and following the same strategy we did to get $\bar{\alpha}$, we can similarly get an invertible arrow $\bar{\beta}: B \otimes C \xrightarrow{ \sim } C \otimes B$ and the following commutative diagrams
\begin{equation} \label{diag.eq21}
\xymatrix{
B \otimes \mathscr{F}D \ar[rrrr]^{id_{\!_{B}} \otimes \mathscr{F}f} \ar@/_/[dd]_{\beta_{\!_{\mathscr{F}D}}}^{\sim} \ar@/_/[drr]_{id_{\!_{B}} \otimes \phi_{\!_{D}}} 
&&&& B \otimes \mathscr{F}D'  \ar@/^/[dll]^{id_{\!_{B}} \otimes \phi_{\!_{D'}}} \ar@/^/[dd]^{\beta_{\!_{\mathscr{F}D'}}}_{\sim} \\
& & B \otimes C \ar@{.>}[d]^(.5){\exists ! \bar{\beta}} \\
\mathscr{F}D \otimes B \ar@/_3pc/[rrrr]_{\mathscr{F}f \otimes id_{\!_{B}}} \ar@/_.3pc/[rr]_(.55){\phi_{\!_{D}} \otimes id_{\!_{B}}} & & C \otimes B && \mathscr{F}D' \otimes B \ar@/^.3pc/[ll]^(.55){\phi_{\!_{D'}} \otimes id_{\!_{B}}}
}
\end{equation}

\vspace{.2cm}

\begin{equation} \label{diag.eq22}
\xymatrix{
\mathscr{F}D \otimes B\ar[rrrr]^{\mathscr{F}f  \otimes id_{\!_{B}}} \ar@/_/[dd]_{\beta_{\!_{\mathscr{F}D}}^{-1}}^{\sim} \ar@/_/[drr]_{\phi_{\!_{D}}\otimes id_{\!_{B}}} 
&&&& \mathscr{F}D' \otimes B \ar@/^/[dll]^{\phi_{\!_{D'}}\otimes id_{\!_{B}}} \ar@/^/[dd]^{\beta_{\!_{\mathscr{F}D'}}^{-1}}_{\sim} \\
& & C \otimes B \ar@{.>}[d]^(.5){\exists ! \bar{\beta}'} \\
B \otimes \mathscr{F}D \ar@/_3pc/[rrrr]_{id_{\!_{B}}\otimes \mathscr{F}f} \ar@/_.3pc/[rr]_(.55){id_{\!_{B}}\otimes \phi_{\!_{D}} } & & B \otimes C && B \otimes \mathscr{F}D' \ar@/^.3pc/[ll]^(.55){id_{\!_{B}}\otimes \phi_{\!_{D'}} }
}
\end{equation} \\

To show that $(C,\bar{\alpha},\bar{\beta}) \in \mathcal{Z}_h(\Cc)$, we need to show that the following diagram is commutative.\\
\begin{equation} \label{diag.eq23}
\xymatrix{
A \otimes C \ar[d]_{\bar{\alpha}}^{\sim} \ar[rr]^-{(h \otimes id_{\!_{C}})}
&& B \otimes C \ar[d]^{\bar{\beta}}_{\sim} \\
C \otimes A \ar[rr]_-{(id_{\!_{C}} \otimes h)} && C \otimes B }
\end{equation}\\

To show this, consider the following diagram\\
\begin{equation} \label{diag.eq24}
\xymatrix{
A \otimes \mathscr{F}D \ar[rrrr]^{id_{\!_{A}} \otimes \mathscr{F}f} \ar@/_/[dd]_{\alpha_{\!_{\mathscr{F}D}}}^{\sim} \ar@/_/[drr]_{id_{\!_{A}} \otimes \phi_{\!_{D}}} 
&&&& A \otimes \mathscr{F}D'  \ar@/^/[dll]^{id_{\!_{A}} \otimes \phi_{\!_{D'}}} \ar@/^/[dd]^{\alpha_{\!_{\mathscr{F}D'}}}_{\sim} \\
& & A \otimes C \ar@/_/[ddl]_{\bar{\alpha}} 
\ar@/^/[ddr]^{h \otimes id_{\!_{C}}} \\
\mathscr{F}D \otimes A \ar@/_2pc/[rrrr]^{\mathscr{F}f \otimes id_{\!_{A}}}
\ar@/_/[dr]_{\phi_{\!_{D}} \otimes id_{\!_{A}}} \ar@/_7pc/[ddrr]_(.4){\phi_{\!_{D}} \otimes h}
&&&& \mathscr{F}D' \otimes A \ar@/^/[dl]^{\phi_{\!_{D'}} \otimes id_{\!_{A}}}\ar@/^7pc/[ddll]^(.4){\phi_{\!_{D'}} \otimes h}
\\
&  C \otimes A \ar@/_/[dr]_{id_{\!_{C}} \otimes h}
& & B \otimes C \ar@/^/[dl]^{\bar{\beta}}
\\ & & C \otimes B 
}
\end{equation}\\

 We note that $(C \otimes A,((id_{\!_{C}} \otimes h)\bar{\alpha} (id_{\!_{A}} \otimes \phi_{\!_{D}}))_{\!_{{D} \in \mathscr{D}}})$ is a cocone on $\mathcal{P}_A$ $ \mathscr{U} \mathscr{F}$ since 
 
 \begin{center}
 •$ (id_{\!_{C}} \otimes h)\bar{\alpha} (id_{\!_{A}} \otimes \phi_{\!_{D'}})(id_{\!_{A}} \otimes \mathscr{F}f) =  (id_{\!_{C}} \otimes h) \bar{\alpha} (id_{\!_{A}} \otimes \phi_{\!_{D'}}\mathscr{F}f) =  (id_{\!_{C}} \otimes h) \bar{\alpha} (id_{\!_{A}} \otimes \phi_{\!_{D}})$
 \end{center}   
 for any $D$ in $\mathscr{D}$. Furthermore, we have \\

\begin{tabular}{lllll}
$(id_{\!_{C}} \otimes h) \bar{\alpha} (id_{\!_{A}} \otimes \phi_{\!_{D}}) $  &  $= (id_{\!_{C}} \otimes h) (\phi_{\!_{D}} \otimes id_{\!_{A}}) \alpha_{\!_{\mathscr{F}D}} $\\
&  (by \ref{diag.eq19})\\
 & $= (\phi_{\!_{D}} \otimes h) \alpha_{\!_{\mathscr{F}D}} $\\
&  (naturality of $\otimes$) \\
 & $= (\phi_{\!_{D}} \otimes id_{\!_{B}}) (id_{\!_{\mathscr{F}D}} \otimes h)\alpha_{\!_{\mathscr{F}D}} $\\
&  (naturality of $\otimes$) \\
& $= (\phi_{\!_{D}} \otimes id_{\!_{B}}) \beta_{\!_{\mathscr{F}D}} (h \otimes id_{\!_{\mathscr{F}D}}) $\\
&  (since $(\mathscr{F}D,\alpha_{\!_{\mathscr{F}D}}, \beta_{\!_{\mathscr{F}D}}) \in \mathcal{Z}_h(\Cc), \forall D \in \mathscr{D}$) \\
& $= \bar{\beta} (id_{\!_{B}} \otimes \phi_{\!_{D}}) (h \otimes id_{\!_{\mathscr{F}D}}) $\\
&  (by \ref{diag.eq21})\\
& $= \bar{\beta} (h \otimes \phi_{\!_{D}}) $\\
&  (naturality of $\otimes$) \\
& $= \bar{\beta} (h \otimes id_{\!_{C}})(id_{\!_{A}} \otimes \phi_{\!_{D}}) $ \\
&  (naturality of $\otimes$) \\
  \end{tabular}\\
  
\vspace{.2cm}

Since  $(A \otimes C,(id_{\!_{A}} \otimes \phi_{\!_{D}})_{\!_{{D} \in \mathscr{D}}})$ is a colimit of $\mathcal{P}_A$ $ \mathscr{U} \mathscr{F}$, we must have $(id_{\!_{C}} \otimes h) \bar{\alpha}  =  \bar{\beta} (h \otimes id_{\!_{C}})$.\\
Therefore, $(C,\bar{\alpha},\bar{\beta}) \in \mathcal{Z}_h(\Cc)$ and $\phi_{\!_{D}}$ is an arrow in $\mathcal{Z}_h(\Cc), \forall D \in \mathscr{D}$. Thus,  $((C,\bar{\alpha},\bar{\beta}),(\phi_{\!_{D}})_{\!_{{D} \in \mathscr{D}}})$ is a cocone on $\mathscr{F}$. To show that  $((C,\bar{\alpha},\bar{\beta}),(\phi_{\!_{D}})_{\!_{{D} \in \mathscr{D}}})$ is a colimit of $\mathscr{F}$, let $((C',\lambda,\gamma),(\psi_{\!_{D}})_{\!_{{D} \in \mathscr{D}}})$ be a cocone on $\mathscr{F}$. Since $(C,(\phi_{\!_{D}})_{\!_{{D} \in \mathscr{D}}})$ is a colimit of $\mathscr{U} \mathscr{F}$, there exists a unique morphism $g: C \rightarrow C'$ in $\mathcal{C}$ with $g\phi_{\!_{D}} = \psi_{\!_{D}}$ for every $D \in \mathscr{D}$. The proof is complete whence we show that $g$ is a morphism in $\mathcal{Z}_h(\Cc)$. Explicitly, we need to show that the diagrams
 
\begin{equation} \label{diag.eq25}
\xymatrix{
A \otimes C \ar[d]_-{(id_{\!_{A}} \otimes g)} \ar[rr]^{\bar{\alpha}}_{\sim}
&& C \otimes A \ar[d]^-{ (g \otimes id_{\!_{A}})} \\
A \otimes C' \ar[rr]_{\lambda}^{\sim} && C' \otimes A }
\hspace{35pt} 
\xymatrix{
B \otimes C \ar[d]_-{(id_{\!_{B}} \otimes g)} \ar[rr]^{\bar{\beta}}_{\sim}
&& C \otimes B \ar[d]^-{ (g \otimes id_{\!_{B}})} \\
B \otimes C' \ar[rr]_{\gamma}^{\sim} && C' \otimes B }
\end{equation}\\
commute. Consider the diagram 
\begin{equation} \label{diag.eq26}
\xymatrix{
A \otimes \mathscr{F}D \ar[rrrr]^{id_{\!_{A}} \otimes \mathscr{F}f} \ar@/_/[dd]_{\alpha_{\!_{\mathscr{F}D}}}^{\sim} \ar@/_/[drr]_{id_{\!_{A}} \otimes \phi_{\!_{D}}} 
&&&& A \otimes \mathscr{F}D'  \ar@/^/[dll]^{id_{\!_{A}} \otimes \phi_{\!_{D'}}} \ar@/^/[dd]^{\alpha_{\!_{\mathscr{F}D'}}}_{\sim} \\
& & A \otimes C \ar@/_/[ddl]_{\bar{\alpha}} 
\ar@/^/[ddr]^{id_{\!_{A}} \otimes g} \\
\mathscr{F}D \otimes A \ar@/_2pc/[rrrr]^{\mathscr{F}f \otimes id_{\!_{A}}}
\ar@/_/[dr]_{\phi_{\!_{D}} \otimes id_{\!_{A}}} \ar@/_7pc/[ddrr]_(.4){\psi_{\!_{D}} \otimes id_{\!_{A}}}
&&&& \mathscr{F}D' \otimes A \ar@/^/[dl]^{\phi_{\!_{D'}} \otimes id_{\!_{A}}}\ar@/^7pc/[ddll]^(.4){\psi_{\!_{D'}} \otimes id_{\!_{A}}}
\\
&  C \otimes A \ar@/_/[dr]_{g \otimes id_{\!_{A}}}
& & A \otimes C' \ar@/^/[dl]^{\lambda}
\\ & & C' \otimes A 
}
\end{equation}\\

Notably, $(C' \otimes A,((g \otimes id_{\!_{A}}) \bar{\alpha} (id_{\!_{A}} \otimes \phi_{\!_{D}}))_{\!_{{D} \in \mathscr{D}}})$ is a cocone on $\mathcal{P}_A$ $ \mathscr{U} \mathscr{F}$  since for every $D \in \mathscr{D}$, we have  
 
 \begin{center}
 •$ (g \otimes id_{\!_{A}})\bar{\alpha} (id_{\!_{A}} \otimes \phi_{\!_{D'}})(id_{\!_{A}} \otimes \mathscr{F}f) =  (g \otimes id_{\!_{A}}) \bar{\alpha} (id_{\!_{A}} \otimes \phi_{\!_{D'}}\mathscr{F}f) =  (g \otimes id_{\!_{A}}) \bar{\alpha} (id_{\!_{A}} \otimes \phi_{\!_{D}})$
 \end{center}
 
\vspace{.2cm}

We also  have  \\

\begin{tabular}{lllll}
$(g \otimes id_{\!_{A}}) \bar{\alpha} (id_{\!_{A}} \otimes \phi_{\!_{D}}) $  &  $=(g \otimes id_{\!_{A}}) (\phi_{\!_{D}} \otimes id_{\!_{A}}) \alpha_{\!_{\mathscr{F}D}} $\\
&  (by \ref{diag.eq19})\\
 & $= (g \phi_{\!_{D}} \otimes id_{\!_{A}}) \alpha_{\!_{\mathscr{F}D}} $\\
&  (naturality of $\otimes$) \\
 &$= (\psi_{\!_{D}} \otimes id_{\!_{A}}) \alpha_{\!_{\mathscr{F}D}} $\\
&  (since $g\phi_{\!_{D}} = \psi_{\!_{D}}$, $\forall D \in \mathscr{D}$) \\
& $= \lambda (id_{\!_{A}} \otimes \psi_{\!_{D}}) $\\
&  (since $\psi_{\!_{D}}$ is a morphism in $\mathcal{Z}_h(\Cc)$,  $\forall D \in \mathscr{D}$) \\
& $= \lambda (id_{\!_{A}} \otimes g\phi_{\!_{D}}) $\\
&   (since $g\phi_{\!_{D}} = \psi_{\!_{D}}$, $\forall D \in \mathscr{D}$) \\
& $= \lambda (id_{\!_{A}} \otimes g) (id_{\!_{A}} \otimes \phi_{\!_{D}}) $\\
&  (naturality of $\otimes$) \\
  \end{tabular}\\
  
  for any $D$ in $\mathscr{D}$. Since $(A \otimes C,(id_{\!_{A}} \otimes \phi_{\!_{D}})_{\!_{{D} \in \mathscr{D}}})$ is a colimit of $\mathcal{P}_A$ $ \mathscr{U} \mathscr{F}$, it follows   
 \begin{equation} \label{diag.eq27}
(g \otimes id_{\!_{A}}) \bar{\alpha}  = \lambda (id_{\!_{A}} \otimes g) .
\end{equation}\\

\vspace{.2cm}

Similarly, replacing $A$ by $B$ and considering the following diagram

\begin{equation} \label{diag.eq28}
\xymatrix{
B \otimes \mathscr{F}D \ar[rrrr]^{id_{\!_{B}} \otimes \mathscr{F}f} \ar@/_/[dd]_{\beta_{\!_{\mathscr{F}D}}}^{\sim} \ar@/_/[drr]_{id_{\!_{B}} \otimes \phi_{\!_{D}}} 
&&&& B \otimes \mathscr{F}D'  \ar@/^/[dll]^{id_{\!_{B}} \otimes \phi_{\!_{D'}}} \ar@/^/[dd]^{\beta_{\!_{\mathscr{F}D'}}}_{\sim} \\
& & B \otimes C \ar@/_/[ddl]_{\bar{\beta}} 
\ar@/^/[ddr]^{id_{\!_{B}} \otimes g} \\
\mathscr{F}D \otimes B \ar@/_2pc/[rrrr]^{\mathscr{F}f \otimes id_{\!_{B}}}
\ar@/_/[dr]_{\phi_{\!_{D}} \otimes id_{\!_{B}}} \ar@/_7pc/[ddrr]_(.4){\psi_{\!_{D}} \otimes id_{\!_{B}}}
&&&& \mathscr{F}D' \otimes B \ar@/^/[dl]^{\phi_{\!_{D'}} \otimes id_{\!_{B}}}\ar@/^7pc/[ddll]^(.4){\psi_{\!_{D'}} \otimes id_{\!_{B}}}
\\
&  C \otimes B \ar@/_/[dr]_{g \otimes id_{\!_{B}}}
& & B \otimes C' \ar@/^/[dl]^{\gamma}
\\ & & C' \otimes B 
}
\end{equation}\\

gives us  
\begin{equation} \label{diag.eq29}
(g \otimes id_{\!_{B}}) \bar{\beta}  = \gamma (id_{\!_{B}} \otimes g) .
\end{equation}\\ 

From \ref{diag.eq27} and \ref{diag.eq29}, we have $g$ is a morphism in $\mathcal{Z}_h(\Cc)$. Thus,  $((C,\bar{\alpha},\bar{\beta}),(\phi_{\!_{D}})_{\!_{{D} \in \mathscr{D}}})$ is a colimit of  $\mathscr{F}$,   and the proof is complete.
\end{proof}

\begin{proposition}\label{p.cocomp2} 
Let $\mathcal{C}$  be a cocomplete category  and $X$  an object in $\Cc$. If $\mathcal{P}_X,\mathcal{Q}_X$ are cocontinuous, then $\mathcal{Z}_X(\Cc)$ is cocomplete and the forgetful functor $\mathscr{U}: \mathcal{Z}_X(\Cc) \rightarrow \mathcal{C}$ is cocontinuous. Moreover, the colimit of objects in $\mathcal{Z}_X(\Cc)$ can be obtained by the corresponding construction for objects in $\Cc$.
\end{proposition}
\begin{proof}
Let $\mathscr{D}$ be a small category, and let $\mathscr{F}:\mathscr{D} \rightarrow \mathcal{Z}_X(\Cc)$ be a functor. Since $\mathcal{C}$ is a cocomplete category, $\mathscr{U} \mathscr{F}$ has a colimit $(C,(\phi_{\!_{D}})_{\!_{{D} \in \mathscr{D}}})$. Since  $\mathcal{P}_X$ is cocontinuous, $(\mathcal{P}_X(C),(\mathcal{P}_X(\phi_{\!_{D}}))_{\!_{{D} \in \mathscr{D}}})$ is a colimit of the functor $\mathcal{P}_X$ $ \mathscr{U} \mathscr{F}$.
Thus, $(X \otimes C,(id_{\!_{X}} \otimes \phi_{\!_{D}})_{\!_{{D} \in \mathscr{D}}})$ is a colimit of $\mathcal{P}_X$ $ \mathscr{U} \mathscr{F}$. \\
First, we note that the functor $\mathscr{F}:\mathscr{D} \rightarrow \mathcal{Z}_X(\Cc)$ assigns to each object $D \in \mathscr{D}$ an object $(\mathscr{F}D,\alpha_{\!_{\mathscr{F}D}}) \in \mathcal{Z}_X(\Cc)$. In addition, we have $\mathscr{F}f: (\mathscr{F}D,\alpha_{\!_{\mathscr{F}D}}) \rightarrow (\mathscr{F}D',\alpha_{\!_{\mathscr{F}D'}})$ is an arrow in $\mathcal{Z}_X(\Cc)$ for every arrow $f: D \rightarrow D'$ in $\mathscr{D}$. By \ref{diag.eq10}, we have the following commutative diagram 

\begin{equation} \label{diag.eq30}
\xymatrix{
\mathscr{F}D \otimes X \ar[d]_-{\alpha_{\!_{\mathscr{F}D}}}^{\sim} \ar[rr]^-{(\mathscr{F}f \otimes id_{\!_{X}})}
& & \mathscr{F}D' \otimes X \ar[d]^-{ \alpha_{\!_{\mathscr{F}D'}}}_{\sim} \\
X \otimes \mathscr{F}D \ar[rr]_-{(id_{\!_{X}} \otimes \mathscr{F}f)} & & X \otimes \mathscr{F}D'}
\end{equation}\\

Consider the diagram
\begin{equation} \label{diag.eq31}
\xymatrix{
\mathscr{F}D \otimes X\ar[rrrr]^{\mathscr{F}f  \otimes id_{\!_{X}}} \ar@/_/[dd]_{\alpha_{\!_{\mathscr{F}D}}}^{\sim} \ar@/_/[drr]_{\phi_{\!_{D}}\otimes id_{\!_{X}}} 
&&&& \mathscr{F}D' \otimes X \ar@/^/[dll]^{\phi_{\!_{D'}}\otimes id_{\!_{X}}} \ar@/^/[dd]^{\alpha_{\!_{\mathscr{F}D'}}}_{\sim} \\
& & C \otimes X \ar@{.>}[d]^(.5){\exists ! \mu_{\!_{X}}} \\
X \otimes \mathscr{F}D \ar@/_3pc/[rrrr]_{id_{\!_{X}}\otimes \mathscr{F}f} \ar@/_.3pc/[rr]_(.55){id_{\!_{X}}\otimes \phi_{\!_{D}} } & & X \otimes C && X \otimes \mathscr{F}D' \ar@/^.3pc/[ll]^(.55){id_{\!_{X}}\otimes \phi_{\!_{D'}} }
}
\end{equation} 

Since  $\mathcal{Q}_X$ is cocontinuous, $(\mathcal{Q}_X (C),(\mathcal{Q}_X (\phi_{\!_{D}}))_{\!_{{D} \in \mathscr{D}}})$ is a colimit of $\mathcal{Q}_X$ $ \mathscr{U} \mathscr{F}$.  
So $(C\otimes X, (\phi_{\!_{D}} \otimes id_{\!_{X}})_{\!_{{D} \in \mathscr{D}}})$ is a colimit of $\mathcal{Q}_X$ $ \mathscr{U} \mathscr{F}$. We also have 

\vspace{.2cm}

\begin{tabular}{lllll}
$(id_{\!_{X}}\otimes \phi_{\!_{D'}})\alpha_{\!_{\mathscr{F}D'}} (\mathscr{F}f \otimes id_{\!_{X}}) $  &  $= (id_{\!_{X}}\otimes \phi_{\!_{D'}}) (id_{\!_{X}} \otimes \mathscr{F}f) \alpha_{\!_{\mathscr{F}D}} $\\
&  (by \ref{diag.eq30})\\
 & $= (id_{\!_{X}}  \otimes \phi_{\!_{D'}} \mathscr{F}f)  \alpha_{\!_{\mathscr{F}D}} $\\
&  (naturality of $\otimes$) \\
 & $= (id_{\!_{X}}  \otimes \phi_{\!_{D}}) \alpha_{\!_{\mathscr{F}D}} $\\
&  (since $(C,(\phi_{\!_{D}})_{\!_{{D} \in \mathscr{D}}})$) is a cocone on $\mathscr{U} \mathscr{F}$) \\
  \end{tabular}\\
  
\vspace{.2cm}
for any $D$ in $\mathscr{D}$. Hence, $(X \otimes C,((id_{\!_{X}}  \otimes \phi_{\!_{D}}) \alpha_{\!_{\mathscr{F}D}})_{\!_{{D} \in \mathscr{D}}})$ is a cocone on $\mathcal{P}_X$ $ \mathscr{U} \mathscr{F}$. Since  $(C \otimes X,( \phi_{\!_{D}} \otimes id_{\!_{X}} )_{\!_{{D} \in \mathscr{D}}})$ is a colimit of $\mathcal{Q}_X$ $ \mathscr{U} \mathscr{F}$, there exits a unique arrow $\mu_{\!_{X}}: C \otimes X \rightarrow X \otimes C$ in $\Cc$ with $\mu_{\!_{X}} (\phi_{\!_{D}} \otimes id_{\!_{X}}) = (id_{\!_{X}} \otimes \phi_{\!_{D}}) \alpha_{\!_{\mathscr{F}D}} $.\\

Similarly, we consider the diagram
\begin{equation} \label{diag.eq32}
\xymatrix{
X \otimes \mathscr{F}D \ar[rrrr]^{id_{\!_{X}} \otimes \mathscr{F}f} \ar@/_/[dd]_{\alpha_{\!_{\mathscr{F}D}}^{-1}}^{\sim} \ar@/_/[drr]_{id_{\!_{X}} \otimes \phi_{\!_{D}}} 
&&&& X \otimes \mathscr{F}D'  \ar@/^/[dll]^{id_{\!_{X}} \otimes \phi_{\!_{D'}}} \ar@/^/[dd]^{\alpha_{\!_{\mathscr{F}D'}}^{-1}}_{\sim} \\
& & X \otimes C \ar@{.>}[d]^(.5){\exists ! \nu_{\!_{X}}} \\
\mathscr{F}D \otimes X \ar@/_3pc/[rrrr]_{\mathscr{F}f \otimes id_{\!_{X}}} \ar@/_.3pc/[rr]_(.55){\phi_{\!_{D}} \otimes id_{\!_{X}}} & & C \otimes X && \mathscr{F}D' \otimes X \ar@/^.3pc/[ll]^(.55){\phi_{\!_{D'}} \otimes id_{\!_{X}}}
}
\end{equation}

\vspace{.2cm}

We notice that 

 \vspace{.3cm}
 
\begin{tabular}{lllll}
$(\phi_{\!_{D'}} \otimes id_{\!_{X}})\alpha_{\!_{\mathscr{F}D'}}^{-1} (id_{\!_{X}} \otimes \mathscr{F}f) $  &  $= (\phi_{\!_{D}} \otimes id_{\!_{X}}) (\mathscr{F}f \otimes id_{\!_{X}}) \alpha_{\!_{\mathscr{F}D}}^{-1} $\\
&  (by \ref{diag.eq30})\\
 & $= (\phi_{\!_{D'}} \mathscr{F}f \otimes id_{\!_{X}})  \alpha_{\!_{\mathscr{F}D}}^{-1} $\\
&  (naturality of $\otimes$) \\
 & $= (\phi_{\!_{D}} \otimes id_{\!_{X}})  \alpha_{\!_{\mathscr{F}D}}^{-1} $\\
&  (since $(C,(\phi_{\!_{D}})_{\!_{{D} \in \mathscr{D}}})$) is a cocone on $\mathscr{U} \mathscr{F}$) \\
  \end{tabular}\\
  
\vspace{.2cm}
  
  for any $D$ in $\mathscr{D}$. Therefore, $(C \otimes X,((\phi_{\!_{D}} \otimes id_{\!_{X}})  \alpha_{\!_{\mathscr{F}D}}^{-1})_{\!_{{D} \in \mathscr{D}}})$ is a cocone on $\mathcal{P}_X$ $ \mathscr{U} \mathscr{F}$. Since  $(X \otimes C,(id_{\!_{X}} \otimes \phi_{\!_{D}})_{\!_{{D} \in \mathscr{D}}})$ is a colimit of $\mathcal{P}_X$ $ \mathscr{U} \mathscr{F}$, there exits a unique arrow $\nu_{\!_{X}}: X \otimes C \rightarrow C \otimes X$ in $\Cc$ with $\nu_{\!_{X}} (id_{\!_{X}} \otimes \phi_{\!_{D}}) = (\phi_{\!_{D}} \otimes id_{\!_{X}}) \alpha_{\!_{\mathscr{F}D}}^{-1} $.\\

Next, we show that $\mu_{\!_{X}}$ is an invertible arrow. 
From the commutativity of the diagrams \ref{diag.eq31} and \ref{diag.eq32}, we have 
\begin{center}
•$ (id_{\!_{X}} \otimes \phi_{\!_{D}}) \alpha_{\!_{\mathscr{F}D}} = \mu_{\!_{X}} (\phi_{\!_{D}} \otimes id_{\!_{X}}) \Leftrightarrow  (id_{\!_{X}} \otimes \phi_{\!_{D}}) = \mu_{\!_{X}} (\phi_{\!_{D}} \otimes id_{\!_{X}})  \alpha_{\!_{\mathscr{F}D}}^{-1} \Leftrightarrow  (id_{\!_{X}} \otimes \phi_{\!_{D}}) = \mu_{\!_{X}} \nu_{\!_{X}} (id_{\!_{X}} \otimes \phi_{\!_{D}})$. \\
\end{center}

Obviously, $(X\otimes C, (\mu_{\!_{X}} \nu_{\!_{X}} (id_{\!_{X}} \otimes \phi_{\!_{D}}))_{\!_{{D} \in \mathscr{D}}})$ is a cocone on $\mathcal{P}_X$ $ \mathscr{U} \mathscr{F}$. Since $(X\otimes C, ( id_{\!_{X}} \otimes \phi_{\!_{D}})_{\!_{{D} \in \mathscr{D}}})$ is a colimit of $\mathcal{P}_X$ $ \mathscr{U} \mathscr{F}$, we have $\mu_{\!_{X}} \nu_{\!_{X}} =   id_{\!_{X \otimes C}}$.\\

In a similar way, from the commutativity of the diagrams \ref{diag.eq31} and \ref{diag.eq32}, we have 
\begin{center}
•$\nu_{\!_{X}} (id_{\!_{X}} \otimes \phi_{\!_{D}}) = (\phi_{\!_{D}} \otimes id_{\!_{X}}) \alpha_{\!_{\mathscr{F}D}}^{-1} \Leftrightarrow \nu_{\!_{X}} (id_{\!_{X}} \otimes \phi_{\!_{D}}) \alpha_{\!_{\mathscr{F}D}} = (\phi_{\!_{D}} \otimes id_{\!_{X}}) \Leftrightarrow \nu_{\!_{X}} \mu_{\!_{X}} (\phi_{\!_{D}} \otimes id_{\!_{X}}) = (\phi_{\!_{D}} \otimes id_{\!_{X}})$. \\
\end{center}

Clearly, $(C\otimes X, (\nu_{\!_{X}} \mu_{\!_{X}} (\phi_{\!_{D}} \otimes id_{\!_{X}}))_{\!_{{D} \in \mathscr{D}}})$ is a cocone on $\mathcal{Q}_X$ $ \mathscr{U} \mathscr{F}$. Since $(C\otimes X, (\phi_{\!_{D}} \otimes id_{\!_{X}})_{\!_{{D} \in \mathscr{D}}})$ is a colimit of $\mathcal{Q}_X$ $ \mathscr{U} \mathscr{F}$, we have $\nu_{\!_{X}} \mu_{\!_{X}} =   id_{\!_{C \otimes X}}$.\\

 Therefore, the arrow $\mu_{\!_{X}}$ is invertible and  $\mu_{\!_{X}}^{-1} =  \nu_{\!_{X}}$. It follows that $(C,\mu_{\!_{X}}) \in \mathcal{Z}_X(\Cc) $, and $\phi_{\!_{D}}$ is an arrow in $\mathcal{Z}_X(\Cc), \forall D \in \mathscr{D}$.  Hence, $((C,\mu_{\!_{X}}),(\phi_{\!_{D}})_{\!_{{D} \in \mathscr{D}}})$ is a cocone on $\mathscr{F}$. \\

It remains to show that $((C,\mu_{\!_{X}}),(\phi_{\!_{D}})_{\!_{{D} \in \mathscr{D}}})$ is a colimit of $\mathscr{F}$, let $((C',\eta),(\psi_{\!_{D}})_{\!_{{D} \in \mathscr{D}}})$ be a cocone on $\mathscr{F}$. Since $(C,(\phi_{\!_{D}})_{\!_{{D} \in \mathscr{D}}})$ is a colimit of $\mathscr{U} \mathscr{F}$, there exists a unique morphism $g: C \rightarrow C'$ in $\mathcal{C}$ with $g\phi_{\!_{D}} = \psi_{\!_{D}}$ for every $D \in \mathscr{D}$.\\

Clearly, all we need is to show that $g$ is a morphism in $\mathcal{Z}_X(\Cc)$. Indeed, we need to show that the diagram
 
\begin{equation} \label{diag.eq33}
\xymatrix{
C \otimes X \ar[d]_-{(g \otimes id_{\!_{X}})} \ar[rr]^{\mu_{\!_{X}}}_{\sim}
&& X \otimes C \ar[d]^-{ (id_{\!_{X}} \otimes g)} \\
C' \otimes X \ar[rr]_{\eta_{\!_{X}}}^{\sim} && X \otimes C' }
\end{equation}\\
commutes. \\

Consider the following diagram 
\begin{equation} \label{diag.eq34}
\xymatrix{
\mathscr{F}D \otimes X \ar[rrrr]^{\mathscr{F}f \otimes id_{\!_{X}}} \ar@/_/[dd]_{\alpha_{\!_{\mathscr{F}D}}}^{\sim} \ar@/_/[drr]_{\phi_{\!_{D}} \otimes id_{\!_{X}}} 
&&&& \mathscr{F}D' \otimes X  \ar@/^/[dll]^{ \phi_{\!_{D'}} \otimes id_{\!_{X}}} \ar@/^/[dd]^{\alpha_{\!_{\mathscr{F}D'}}}_{\sim} \\
& & C \otimes X \ar@/_/[ddl]_{\mu_{\!_{X}}} 
\ar@/^/[ddr]^{g \otimes id_{\!_{X}}} \\
X \otimes \mathscr{F}D \ar@/_2pc/[rrrr]^{id_{\!_{X}} \otimes \mathscr{F}f}
\ar@/_/[dr]_{id_{\!_{X}} \otimes \phi_{\!_{D}}} \ar@/_7pc/[ddrr]_(.4){id_{\!_{X}} \otimes \psi_{\!_{D}}}
&&&& X \otimes \mathscr{F}D' \ar@/^/[dl]^{id_{\!_{X}} \otimes \phi_{\!_{D'}}}\ar@/^7pc/[ddll]^(.4){ id_{\!_{X}} \otimes  \psi_{\!_{D'}}}
\\
&  X \otimes C \ar@/_/[dr]_{id_{\!_{X}} \otimes g}
& & C' \otimes X \ar@/^/[dl]^{\eta}
\\ & & X \otimes C' 
}
\end{equation}\\

Notably, $(X \otimes C', ((id_{\!_{X}} \otimes g) \mu_{\!_{X}} (\phi_{\!_{D}} \otimes id_{\!_{X}}))_{\!_{{D} \in \mathscr{D}}})$ is a cocone on $\mathcal{P}_X$ $ \mathscr{U} \mathscr{F}$ since 
 
 \begin{center}
 •$ (id_{\!_{X}} \otimes g) \mu_{\!_{X}} (\phi_{\!_{D'}} \otimes id_{\!_{X}})(\mathscr{F}f \otimes id_{\!_{X}}) =  (id_{\!_{X}} \otimes g) \mu_{\!_{X}} (\phi_{\!_{D'}} \mathscr{F}f \otimes id_{\!_{X}}) = (id_{\!_{X}} \otimes g) \mu_{\!_{X}} (\phi_{\!_{D}} \otimes id_{\!_{X}}) $
 \end{center}
 
\vspace{.2cm}

We also  have  \\

\begin{tabular}{lllll}
$(id_{\!_{X}} \otimes g) \mu_{\!_{X}} (\phi_{\!_{D}} \otimes id_{\!_{X}})$  &  $=(id_{\!_{X}} \otimes g) (id_{\!_{X}} \otimes \phi_{\!_{D}}) \alpha_{\!_{\mathscr{F}D}} $\\
&  (by \ref{diag.eq31})\\
 & $= (id_{\!_{X}} \otimes g \phi_{\!_{D}}) \alpha_{\!_{\mathscr{F}D}} $\\
&  (naturality of $\otimes$) \\
 &$= (id_{\!_{X}} \otimes \psi_{\!_{D}}) \alpha_{\!_{\mathscr{F}D}} $\\
&  (since $g\phi_{\!_{D}} = \psi_{\!_{D}}$, $\forall D \in \mathscr{D}$) \\
& $= \eta (\psi_{\!_{D}} \otimes id_{\!_{X}}) $\\
&  (since $\psi_{\!_{D}}$ is a morphism in $\mathcal{Z}_X(\Cc)$,  $\forall D \in \mathscr{D}$) \\
& $= \eta (g\phi_{\!_{D}} \otimes id_{\!_{X}}) $\\
&   (since $g\phi_{\!_{D}} = \psi_{\!_{D}}$, $\forall D \in \mathscr{D}$) \\
& $= \eta (g \otimes id_{\!_{X}}) (\phi_{\!_{D}} \otimes id_{\!_{X}}) $\\
&  (naturality of $\otimes$) \\
  \end{tabular}\\
  
  \vspace{.2cm}
  
for any $D$ in $\mathscr{D}$. \\

Since $(C \otimes X,(\phi_{\!_{D}} \otimes id_{\!_{X}})_{\!_{{D} \in \mathscr{D}}})$ is a colimit of $\mathcal{Q}_X$ $ \mathscr{U} \mathscr{F}$, we have   
 \begin{equation} \label{diag.eq35}
(id_{\!_{X}} \otimes g) \mu_{\!_{X}} = \eta (g \otimes id_{\!_{X}}).
\end{equation}\\
It follows $((C,\mu_{\!_{X}}),(\phi_{\!_{D}})_{\!_{{D} \in \mathscr{D}}})$ is a colimit of $\mathscr{F}$, and this ends the proof. 

\end{proof}

\begin{remark}\label{r.embeddings}
Let $(\Cc,\otimes,I)$ be monoidal category with $\mathcal{C}$  cocomplete, and let $\mathcal{P}_X$, $\mathcal{Q}_X$ be cocontinuous $\forall X \in \Cc$. Then by Proposition \ref{p.cocomp2},  $\mathcal{Z}_X(\Cc)$ is cocomplete and the forgetful functor $\mathscr{U}: \mathcal{Z}_X(\Cc) \rightarrow \mathcal{C}$ is cocontinuous for all $X \in \Cc$. \\
Let $\mathscr{D}$ be a small category, and let $\mathscr{F}:\mathscr{D} \rightarrow \mathcal{Z}(\Cc)$ be a functor. It follows  from Proposition \ref{p.cocomp2} that $((C,\mu_{\!_{X}}),(\phi_{\!_{D}})_{\!_{{D} \in \mathscr{D}}})$ is a colimit of $\mathscr{H}_{\!_{X}} \mathscr{F}$ (resp. $\mathscr{K}_{\!_{X}} \mathscr{F}$) for all $X \in \Cc$. It turns out that we get a family of invertible arrows $\{\mu_{\!_{X}}\}$,  where $\mu_{\!_{X}}$ is the map in diagram \ref{diag.eq31} for all $X \in \Cc$. 

\end{remark}

\begin{proposition}\label{p.cocomp3} 
Let $(\Cc,\otimes,I)$ be monoidal category with $\mathcal{C}$  cocomplete. If $P_X,\mathcal{Q}_X$ are cocontinuous $\forall X \in \Cc$, then $\mathcal{Z}(\Cc)$ is cocomplete and the forgetful functor $\mathscr{U}: \mathcal{Z}(\Cc) \rightarrow \mathcal{C}$ is cocontinuous. Further, the colimit of objects in $\mathcal{Z}(\Cc)$ can be obtained by the corresponding construction for objects in $\Cc$.
\end{proposition}
\begin{proof}
Let $\mathscr{D}$ be a small category, and let $\mathscr{F}:\mathscr{D} \rightarrow \mathcal{Z}(\Cc)$ be a functor. \\
Since $\mathcal{C}$ is a cocomplete category, the functor $\mathscr{U} \mathscr{F}$ has a colimit $(C,(\phi_{\!_{D}})_{\!_{{D} \in \mathscr{D}}})$. For every $X \in \Cc$, $(\mathcal{P}_X(C),(\mathcal{P}_X(\phi_{\!_{D}}))_{\!_{{D} \in \mathscr{D}}})$ is a colimit of the functor $\mathcal{P}_X$ $ \mathscr{U} \mathscr{F}$ because $\mathcal{P}_X$ is cocontinuous. \\
Thus, for every $X \in \Cc$, $(X \otimes C,(id_{\!_{X}} \otimes \phi_{\!_{D}})_{\!_{{D} \in \mathscr{D}}})$ is a colimit of $\mathcal{P}_X$ $ \mathscr{U} \mathscr{F}$. \\

Further, the functor $\mathscr{F}:\mathscr{D} \rightarrow \mathcal{Z}_X(\Cc)$ assigns to each object $D \in \mathscr{D}$ an object $(\mathscr{F}D,\alpha^{\!^{\mathscr{F}D}}) \in \mathcal{Z}(\Cc)$. Also, we have $\mathscr{F}f: (\mathscr{F}D,\alpha^{\!^{\mathscr{F}D}}) \rightarrow (\mathscr{F}D',\alpha^{\!^{\mathscr{F}D'}})$ is an arrow in $\mathcal{Z}(\Cc)$ for every arrow $f: D \rightarrow D'$ in $\mathscr{D}$. \\
By Remark  \ref{r.Z.X.Z.h}, we have $((C,\mu_{\!_{X}}),(\phi_{\!_{D}})_{\!_{{D} \in \mathscr{D}}})$ is a colimit of $\mathscr{H}_{\!_{X}} \mathscr{F}$, $\forall X \in \Cc$, and thus, we obtain a family of invertible arrows $\{\mu_{\!_{X}}\}_{\!_{X \in \Cc}}$ in $\Cc$, where $\mu_{\!_{X}}$ is the map in diagram \ref{diag.eq31} for all $X \in \Cc$.  Therefore, the proof is complete whence we show that $\{\mu_{\!_{X}}\}_{\!_{X \in \Cc}}$ are natural in $X$, for any $X \in \Cc$, and the conditions (\ref{def.eq4}) and (\ref{def.eq5}) hold.\\

To show that $\mu: C \otimes - \rightarrow - \otimes C$ is a natural transformation, let $\zeta : A \rightarrow B$ be an arrow in $\Cc$. We need to show that the following diagram is commutative

\begin{equation} \label{diag.eq36}
\xymatrix{
C \otimes A \ar[d]_-{(id_{\!_{C}} \otimes \zeta)} \ar[rr]^{\mu_{\!_{A}}}_{\sim}
&& A \otimes C \ar[d]^-{ (\zeta \otimes id_{\!_{C}})} \\
C \otimes B \ar[rr]_{\mu_{\!_{B}}}^{\sim} && B \otimes C }
\end{equation}\\

Since $(\mathscr{F}D,\alpha^{\!^{\mathscr{F}D}}) \in \mathcal{Z}(\Cc)$, $\forall D \in \mathscr{D} $, $\alpha^{\!^{\mathscr{F}D}}$ is a natural transformation, $\forall D \in \mathscr{D} $. Thus, for any $ D \in \mathscr{D} $, the diagram 

\begin{equation} \label{diag.eq37}
\xymatrix{
\mathscr{F}D \otimes A \ar[d]_-{(id_{\!_{\mathscr{F}D}} \otimes \zeta)} \ar[rr]^{\alpha^{\!^{\mathscr{F}D}}_{\!_{A}}}_{\sim}
&& A \otimes \mathscr{F}D \ar[d]^-{ (\zeta \otimes id_{\!_{\mathscr{F}D}})} \\
\mathscr{F}D \otimes B \ar[rr]_{\alpha^{\!^{\mathscr{F}D}}_{\!_{B}}}^{\sim} && B \otimes \mathscr{F}D }
\end{equation}\\
commutes. \\

Now, consider the following diagram

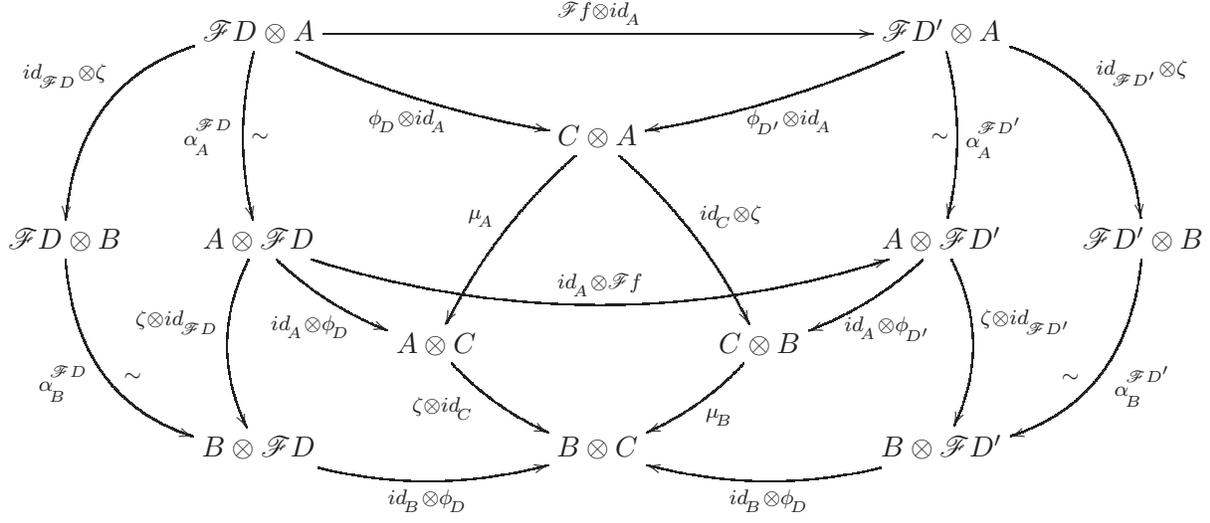
\begin{equation} \label{diag.eq38}
\xymatrix{
& \mathscr{F}D \otimes A \ar[rrrr]^{\mathscr{F}f \otimes id_{\!_{A}}}
\ar@/_2pc/[ddl]_(.5){id_{\!_{\mathscr{F}D}} \otimes \zeta} 
\ar@/_/[dd]_{\alpha^{\!^{\mathscr{F}D}}_{\!_{A}}}^{\sim} \ar@/_/[drr]_{\phi_{\!_{D}} \otimes id_{\!_{A}}} 
&&&& \mathscr{F}D' \otimes A  \ar@/^/[dll]^{ \phi_{\!_{D'}} \otimes id_{\!_{A}}} \ar@/^/[dd]^{\alpha^{\!^{\mathscr{F}D'}}_{\!_{A}}}_{\sim}
\ar@/^2pc/[ddr]^(.5){id_{\!_{\mathscr{F}D'}} \otimes \zeta} & \\
& & & C \otimes A \ar@/_/[ddl]_{\mu_{\!_{A}}} 
\ar@/^/[ddr]^{id_{\!_{C}} \otimes \zeta} &\\
\mathscr{F}D \otimes B \ar@/_2pc/[ddr]_(.4){\alpha^{\!^{\mathscr{F}D}}_{\!_{B}}}^{\sim}
& A \otimes \mathscr{F}D \ar@/_2pc/[rrrr]^{id_{\!_{A}} \otimes \mathscr{F}f}
\ar@/_/[dr]_{id_{\!_{A}} \otimes \phi_{\!_{D}}} \ar@/_1pc/[dd]_(.4){\zeta \otimes id_{\!_{\mathscr{F}D}}}
&&&& A \otimes \mathscr{F}D' \ar@/^/[dl]^{id_{\!_{A}} \otimes \phi_{\!_{D'}}}\ar@/^1pc/[dd]^(.4){ \zeta \otimes id_{\!_{\mathscr{F}D'}}}& \mathscr{F}D' \otimes B \ar@/^2pc/[ddl]^(.4){\alpha^{\!^{\mathscr{F}D'}}_{\!_{B}}}_{\sim}
\\
&&  A \otimes C \ar@/_/[dr]_(.3){\zeta \otimes id_{\!_{C}}}
& & C \otimes B \ar@/^/[dl]^(.4){\mu_{\!_{B}}} &
\\ 
 &B \otimes \mathscr{F}D \ar@/_1pc/[rr]_{id_{\!_{B}} \otimes \phi_{\!_{D}}}
 & & B \otimes C & & B \otimes \mathscr{F}D' \ar@/^1pc/[ll]^{id_{\!_{B}} \otimes \phi_{\!_{D}}}
}
\end{equation}\\

Clearly, $(B \otimes C, ((\zeta \otimes id_{\!_{C}}) \mu_{\!_{A}} (\phi_{\!_{D}} \otimes id_{\!_{A}}))_{\!_{{D} \in \mathscr{D}}})$ is a cocone on $\mathcal{Q}_A$ $ \mathscr{U} \mathscr{F}$ since 

\begin{center}
•$(\zeta \otimes id_{\!_{C}}) \mu_{\!_{A}} (\phi_{\!_{D'}} \otimes id_{\!_{A}})(\mathscr{F} \otimes id_{\!_{A}}) = (\zeta \otimes id_{\!_{C}}) \mu_{\!_{A}} (\phi_{\!_{D'}} \mathscr{F} \otimes id_{\!_{A}}) = (\zeta \otimes id_{\!_{C}}) \mu_{\!_{A}} (\phi_{\!_{D}} \otimes id_{\!_{A}})$ 
\end{center}

\vspace{.2cm}

Furthermore, we have \\

\begin{tabular}{lllll}
$(\zeta \otimes id_{\!_{C}}) \mu_{\!_{A}} (\phi_{\!_{D}} \otimes id_{\!_{A}})$  &  $ =(\zeta \otimes id_{\!_{C}}) (id_{\!_{A}} \otimes \phi_{\!_{D}}) \alpha^{\!^{\mathscr{F}D}}_{\!_{A}}$ \\
&  (by (\ref{diag.eq31}) taking $A$ and $\alpha^{\!^{\mathscr{F}D}}_{\!_{A}}$ in place of $X$ $\alpha_{\!_{\mathscr{F}D}}$ respectively)\\
 &  $ =(\zeta id_{\!_{A}} \otimes id_{\!_{C}} \phi_{\!_{D}}) \alpha^{\!^{\mathscr{F}D}}_{\!_{A}}$\\
&  (naturality of $\otimes$) \\
 &  $ =(\zeta  \otimes \phi_{\!_{D}}) \alpha^{\!^{\mathscr{F}D}}_{\!_{A}}$\\
 &$= (id_{\!_{B}} \otimes \phi_{\!_{D}}) (\zeta \otimes id_{\!_{\mathscr{F}D}}) \alpha^{\!^{\mathscr{F}D}}_{\!_{A}}$\\
 &  (naturality of $\otimes$) \\
 
 &$= (id_{\!_{B}} \otimes \phi_{\!_{D}}) \alpha^{\!^{\mathscr{F}D}}_{\!_{B}} (id_{\!_{\mathscr{F}D}} \otimes \zeta) $\\
 &  (since $\alpha^{\!^{\mathscr{F}D}}$ is a natural transformation, $\forall D \in \mathscr{D} $) \\
 
  & $= \mu_{\!_{B}} (\phi_{\!_{D}} \otimes id_{\!_{B}}) (id_{\!_{\mathscr{F}D}} \otimes \zeta) $\\
 &  (by (\ref{diag.eq31}) taking $B$ and $\alpha^{\!^{\mathscr{F}D}}_{\!_{B}}$ in place of $X$ $\alpha_{\!_{\mathscr{F}D}}$ respectively)\\
&  $= \mu_{\!_{B}} (\phi_{\!_{D}} id_{\!_{\mathscr{F}D}} \otimes id_{\!_{B}}\zeta) $\\
&   (naturality of $\otimes$) \\
& $= \mu_{\!_{B}} (\phi_{\!_{D}} \otimes \zeta) $\\
& $= \mu_{\!_{B}} (id_{\!_{C}} \otimes \zeta) (\phi_{\!_{D}} \otimes id_{\!_{A}}) $\\
&  (naturality of $\otimes$) \\
  \end{tabular}\\

\vspace{.2cm}

Since  $(C \otimes A,( \phi_{\!_{D}} \otimes id_{\!_{A}} )_{\!_{{D} \in \mathscr{D}}})$ is a colimit of $\mathcal{Q}_A$ $ \mathscr{U} \mathscr{F}$, it follows that  $(\zeta \otimes id_{\!_{C}}) \mu_{\!_{A}} = \mu_{\!_{B}} (id_{\!_{C}} \otimes \zeta)$. Hence, $\mu: C \otimes - \longrightarrow - \otimes C$ is a natural transformation. \\
By Remark \ref{r.Z.X.Z.h}, it remains to show that the condition  (\ref{def.eq5}) holds. Consider the following diagram

\begin{equation} \label{diag.eq39}
\xymatrix{
\mathscr{F}D \otimes X \otimes Y \ar[rrrr]^{\mathscr{F}f \otimes id_{\!_{X \otimes Y}}} \ar@/_/[dd]_{\alpha^{\!^{\mathscr{F}D}}_{\!_{X \otimes Y}}}^{\sim} \ar@/_/[drr]_{\phi_{\!_{D}} \otimes id_{\!_{X \otimes Y}}} 
&&&& \mathscr{F}D' \otimes X \otimes Y  \ar@/^/[dll]^{ \phi_{\!_{D'}} \otimes id_{\!_{X \otimes Y}}} \ar@/^/[dd]^{\alpha^{\!^{\mathscr{F}D'}}_{\!_{X \otimes Y}}}_{\sim} \\
& & C \otimes X \otimes Y \ar@/_2pc/[dddd]_(.48){\mu_{\!_{X \otimes Y}}} 
\ar@/^/[ddr]^(.6){\mu_{\!_{X }} \otimes id_{\!_{Y}}} \\
X \otimes Y \otimes \mathscr{F}D 
\ar@/_5pc/[dddrr]_(.3){id_{\!_{X \otimes Y}} \otimes \phi_{\!_{D}}}
&&&& X \otimes Y \otimes \mathscr{F}D' \ar@/^5pc/[dddll]^(.3){ id_{\!_{X \otimes Y}} \otimes  \phi_{\!_{D'}}}
\\
&  & & X \otimes C \otimes Y \ar@/^/[ddl]^(.4){id_{\!_{X }} \otimes \mu_{\!_{Y}}}
\\ 
&&&&&\\
& & X \otimes Y \otimes C 
}
\end{equation}\\

\vspace{.2cm}

\begin{tabular}{lllll}
$(id_{\!_{X }} \otimes \mu_{\!_{Y}}) (\mu_{\!_{X }} \otimes id_{\!_{Y}}) (\phi_{\!_{D}} \otimes id_{\!_{X \otimes Y}})$  &  $ =(id_{\!_{X }} \otimes \mu_{\!_{Y}}) (\mu_{\!_{X }} \otimes id_{\!_{Y}}) (\phi_{\!_{D}} \otimes id_{\!_{X}}\otimes id_{\!_{Y}})$ \\
&  (since  $\otimes$ is a bifunctor)\\
 &  $ =(id_{\!_{X }} \otimes \mu_{\!_{Y}}) (\mu_{\!_{X }} (\phi_{\!_{D}} \otimes id_{\!_{X}}) \otimes id_{\!_{Y}} id_{\!_{Y}}) $ \\
&  (naturality of $\otimes$) \\
 &  $ =(id_{\!_{X }} \otimes \mu_{\!_{Y}}) ((id_{\!_{X}} \otimes \phi_{\!_{D}}) \alpha^{\!^{\mathscr{F}D}}_{\!_{X}} \otimes id_{\!_{Y}}) $ \\
&  (by (\ref{diag.eq31}) taking $\alpha^{\!^{\mathscr{F}D}}_{\!_{X}}$ in place of $\alpha_{\!_{\mathscr{F}D}}$) \\
 & $ =(id_{\!_{X }} \otimes \mu_{\!_{Y}}) (id_{\!_{X}} \otimes \phi_{\!_{D}} \otimes id_{\!_{Y}}) (\alpha^{\!^{\mathscr{F}D}}_{\!_{X}} \otimes id_{\!_{Y}}) $ \\
 &  (naturality of $\otimes$) \\
 
 & $ =(id_{\!_{X }} id_{\!_{X }} \otimes \mu_{\!_{Y}}(\phi_{\!_{D}} \otimes id_{\!_{Y}})) (\alpha^{\!^{\mathscr{F}D}}_{\!_{X}} \otimes id_{\!_{Y}}) $ \\
 &  (naturality of $\otimes$) \\
 
  & $ =(id_{\!_{X }} \otimes (id_{\!_{Y}} \otimes \phi_{\!_{D}}) \alpha^{\!^{\mathscr{F}D}}_{\!_{Y}} )(\alpha^{\!^{\mathscr{F}D}}_{\!_{X}} \otimes id_{\!_{Y}}) $ \\
 &  (by taking $Y,\alpha^{\!^{\mathscr{F}D}}_{\!_{Y}}$ in (\ref{diag.eq31}) in place of $X,\alpha_{\!_{\mathscr{F}D}}$ respectively)\\
&$  =(id_{\!_{X }} \otimes id_{\!_{Y}} \otimes \phi_{\!_{D}}) (id_{\!_{X }} \otimes  \alpha^{\!^{\mathscr{F}D}}_{\!_{Y}}) (\alpha^{\!^{\mathscr{F}D}}_{\!_{X}} \otimes id_{\!_{Y}}) $ \\
&   (naturality of $\otimes$) \\
& $ =(id_{\!_{X }} \otimes id_{\!_{Y}} \otimes \phi_{\!_{D}})  \alpha^{\!^{\mathscr{F}D}}_{\!_{X \otimes Y}} $ \\
& (since $(\mathscr{F}D,\alpha^{\!^{\mathscr{F}D}}) \in \mathcal{Z}(\Cc)$)\\
& $ =(id_{\!_{X \otimes Y}} \otimes \phi_{\!_{D}}) \alpha^{\!^{\mathscr{F}D}}_{\!_{X \otimes Y}} $ \\
& (naturality of $\otimes$) \\
  \end{tabular}\\

\vspace{.2cm}

Since  $(C \otimes X \otimes Y,( \phi_{\!_{D}} \otimes id_{\!_{ X \otimes Y}} )_{\!_{{D} \in \mathscr{D}}})$ is a colimit of $\mathcal{Q}_{ X \otimes Y}$ $ \mathscr{U} \mathscr{F}$, it follows that the condition (\ref{def.eq5}) is satisfied. Therefore, $(C,\mu)$ is a colimit of $\mathscr{F}$ and the proof is complete. 

\end{proof}

The following is an immediate consequence of the proof of Proposition \ref{p.cocomp3} and Remark \ref{r.Z.X.Z.h}.
\begin{corollary}\label{c.cocomp4} 
Let $\mathcal{C}$ be a cocomplete category. If $\mathcal{P}_X,\mathcal{Q}_X$ are cocontinuous $\forall X \in \Cc$, then $\mathcal{Z}_{\omega}(\Cc)$ is cocomplete and the forgetful functor $\mathscr{U}: \mathcal{Z}_{\omega}(\Cc) \rightarrow \mathcal{C}$ is cocontinuous. Moreover, the colimit of objects in $\mathcal{Z}_{\omega}(\Cc)$ can be obtained by the corresponding construction for objects in $\Cc$.\\
\end{corollary}
\begin{proof}

\end{proof}

\section{\textbf{Co-wellpoweredness}}\label{s.cowell}

Let $\mathfrak{E}$ be a class of all epimorphisms of a category $\mathfrak{A}$. Then $\mathfrak{A}$ is called \textit{co-wellpowered} provided that no $\mathfrak{A}$-object has a proper class of pairwise non-isomorphic quotients \cite[p. 125]{Adamek}.  In other words, for every object the quotients form a set \cite[p. 92, 95]{Schubert}. We refer the reader to \cite{Adamek} basics on quotients and co-wellpowered categories.

\begin{proposition}\label{p.cowell.comon} \cite[p. 5]{Abdulwahid} 
Let $CoMon(\Cc)$ be the category of comonoids of $\Cc$ and $U:CoMon(\Cc)\rightarrow \Cc$ be the forgetful functor. If $\Cc$ is co-wellpowered, then so is $CoMon(\Cc)$. 
\end{proposition}
\begin{proof}

\end{proof}

\begin{proposition}\label{p.cowell5} 
Let  $\mathcal{C}$  be a co-wellpowered category, and let $h: A \rightarrow B$ be an arrow in $\Cc$. If $\mathcal{P}_J$ is cocontinuous $\forall J \in \{ A,B\}$, then $\mathcal{Z}_h(\Cc)$ is co-wellpowered. 
\end{proposition}

\begin{proof}
It is enough to show that if $p:(X,\alpha,\beta) \rightarrow (Y,\alpha',\beta')$ and $q:(X,\alpha,\beta) \rightarrow (Z,\alpha'',\beta'')$ are in $\mathcal{Z}_h(\Cc)$ and equivalent as epimorphisms in $\Cc$, then they are equivalent (as epimorphisms) in $\mathcal{Z}_h(\Cc)$. Let $\theta:Y \rightarrow Z$ be an isomorphism in $\Cc$ for which $\theta p = q$. We show that $\theta$ is in fact an isomorphism in $\mathcal{Z}_h(\Cc)$.

Since $p$ and $q$ are arrows in $\mathcal{Z}_h(\Cc)$, the following diagrams are commutative.
\begin{equation} \label{diag.eq40}
\xymatrix{
A \otimes X \ar[d]_-{(id_{\!_{A}} \otimes p)} \ar[rr]^{\alpha}_{\sim}
&& X \otimes A \ar[d]^-{ (p \otimes id_{\!_{A}})} \\
A \otimes Y \ar[rr]_{\alpha'}^{\sim} && Y \otimes A }
\hspace{35pt} 
\xymatrix{
B \otimes X \ar[d]_-{(id_{\!_{B}} \otimes p)} \ar[rr]^{\beta}_{\sim}
&& X \otimes B \ar[d]^-{ (p \otimes id_{\!_{B}})} \\
B \otimes Y \ar[rr]_{\beta'}^{\sim} && Y \otimes B }
\end{equation}

\begin{equation} \label{diag.eq41}
\xymatrix{
A \otimes X  \ar[d]_-{(id_{\!_{A}} \otimes q)} \ar[rr]^{\alpha}_{\sim}
&& X \otimes A \ar[d]^-{ (q \otimes id_{\!_{A}})} \\
A \otimes Z \ar[rr]_{\alpha''}^{\sim} && Z \otimes A }
\hspace{35pt} 
\xymatrix{
B \otimes X \ar[d]_-{(id_{\!_{B}} \otimes q)} \ar[rr]^{\beta}_{\sim}
&& X \otimes B \ar[d]^-{ (q \otimes id_{\!_{B}})} \\
B \otimes Z \ar[rr]_{\beta''}^{\sim} && Z \otimes B }
\end{equation}

Consider the following diagrams:

\begin{equation} \label{diag.eq42}
\xymatrix{
A \otimes X \ar[d]_-{(id_{\!_{A}} \otimes p)}  \ar@/_5pc/[dd]_{(id_{\!_{A}} \otimes q)} \ar[rr]^{\alpha}_{\sim} && X \otimes A \ar@/^5pc/[dd]^{(q \otimes id_{\!_{A}})} \ar[d]^-{ (p \otimes id_{\!_{A}})} \\
A \otimes Y \ar[d]_-{(id_{\!_{A}} \otimes \theta)} \ar[rr]_{\alpha'}^{\sim} && Y \otimes A \ar[d]^-{ (\theta \otimes id_{\!_{A}})}\\
A \otimes Z \ar[rr]_{\alpha''}^{\sim} && Z \otimes A }
\end{equation}

\begin{equation} \label{diag.eq43}
\xymatrix{
B \otimes X \ar[d]_-{(id_{\!_{B}} \otimes p)}  \ar@/_5pc/[dd]_{(id_{\!_{B}} \otimes q)} \ar[rr]^{\alpha}_{\sim} && X \otimes B \ar@/^5pc/[dd]^{(q \otimes id_{\!_{B}})} \ar[d]^-{ (p \otimes id_{\!_{B}})} \\
B \otimes Y \ar[d]_-{(id_{\!_{B}} \otimes \theta)} \ar[rr]_{\alpha'}^{\sim} && Y \otimes B \ar[d]^-{ (\theta \otimes id_{\!_{B}})}\\
B \otimes Z \ar[rr]_{\alpha''}^{\sim} && Z \otimes B }
\end{equation}

We have\\

\begin{tabular}{lllll}
$\alpha'' (id_{\!_{A}} \otimes \theta) (id_{\!_{A}} \otimes p)$  &  $ = \alpha'' (id_{\!_{A}} \otimes \theta p) $ \\
&  (naturality of $\otimes$)\\
 &  $ = \alpha'' (id_{\!_{A}} \otimes q) $ \\
&  (since  $\theta p = q$  \\
 &  $ = (q \otimes id_{\!_{A}}) \alpha $ \\
&  (by (\ref{diag.eq41})) \\
 & $ = (\theta p \otimes id_{\!_{A}}) \alpha $ \\
 &  (since  $\theta p = q$  \\
 & $ = (\theta \otimes id_{\!_{A}}) (p \otimes id_{\!_{A}}) \alpha $ \\
 & (naturality of $\otimes$)\\
  &  $ = (\theta \otimes id_{\!_{A}}) \alpha' (id_{\!_{A}} \otimes p) $ \\
  &  (by (\ref{diag.eq40})) \\
\end{tabular}\\

\vspace{.2cm}

Since $P_A$ is cocontinuous, it preserves epimorphisms \cite[p. 72]{Mac Lane1}. Hence, $P_A (p) = (id_{\!_{A}} \otimes p)$ is an epimorphism. Thus, $\alpha'' (id_{\!_{A}} \otimes \theta)  = (\theta \otimes id_{\!_{A}}) \alpha'$. Similarly, from diagram (\ref{diag.eq43}), we get $\alpha'' (id_{\!_{B}} \otimes \theta)  = (\theta \otimes id_{\!_{B}}) \alpha'$. Therefore, $\theta$ is an isomorphism in $\mathcal{Z}_h(\Cc)$.

\end{proof}

\begin{proposition}\label{p.cowell6} 
Let  $\mathcal{C}$  be a co-wellpowered category and  $X$  an object in $\Cc$. If $\mathcal{Q}_X$ is cocontinuous, then $\mathcal{Z}_X(\Cc)$ is co-wellpowered. 
\end{proposition}

\begin{proof} 
As in Proposition \ref{p.cowell5}, it suffices to show that if $p:(A,\alpha) \rightarrow (B,\beta)$ and $q:(A,\alpha) \rightarrow (B',\beta')$ are in $\mathcal{Z}_X(\Cc)$ and equivalent as epimorphisms in $\Cc$, then they are equivalent (as epimorphisms) in $\mathcal{Z}_X(\Cc)$. Let $\theta:B \rightarrow B'$ be an isomorphism in $\Cc$ with  $\theta p = q$. We show that $\theta$ is in fact an isomorphism in $\mathcal{Z}_X(\Cc)$.

Since $p$ and $q$ are arrows in $\mathcal{Z}_X(\Cc)$, the following diagrams are commutative.
\begin{equation} \label{diag.eq44}
\xymatrix{
A \otimes X \ar[d]_-{(p \otimes id_{\!_{X}})} \ar[rr]^{\alpha}_{\sim}
&& X \otimes A \ar[d]^-{(id_{\!_{X}} \otimes p)} \\
B \otimes X \ar[rr]_{\beta}^{\sim} && X \otimes B }
\hspace{35pt} 
\xymatrix{
A \otimes X \ar[d]_-{(q \otimes id_{\!_{X}})} \ar[rr]^{\alpha}_{\sim}
&& X \otimes A \ar[d]^-{ (id_{\!_{X}} \otimes q)} \\
B' \otimes X \ar[rr]_{\beta'}^{\sim} && X \otimes B'  }
\end{equation}

Consider the following diagram

\begin{equation} \label{diag.eq45}
\xymatrix{
A \otimes X \ar[d]_-{(p \otimes id_{\!_{X}})}  \ar@/_5pc/[dd]_{(q \otimes id_{\!_{X}})} \ar[rr]^{\alpha}_{\sim} && X \otimes A \ar@/^5pc/[dd]^{(id_{\!_{X}} \otimes q)} \ar[d]^-{ (id_{\!_{X}} \otimes p)} \\
B \otimes X \ar[d]_-{(\theta \otimes id_{\!_{X}})} \ar[rr]_{\beta}^{\sim} &&  X \otimes B \ar[d]^-{ (id_{\!_{X}} \otimes \theta)}\\
B' \otimes X \ar[rr]_{\beta'}^{\sim} &&  X \otimes B' }
\end{equation}\\

We have\\

\begin{tabular}{lllll}
$\beta' (\theta \otimes id_{\!_{X}}) (p \otimes id_{\!_{X}})$  &  $ = \beta' (\theta p \otimes id_{\!_{X}}) $ \\
&  (naturality of $\otimes$)\\
 &  $ = \beta' (q \otimes id_{\!_{X}}) $ \\
&  (since  $\theta p = q$  \\
 &  $ = (id_{\!_{X}} \otimes q) \alpha $ \\
&  (by (\ref{diag.eq44})) \\
 & $ =  (id_{\!_{X}} \otimes \theta p) \alpha $ \\
 &  (since  $\theta p = q$  \\
 & $ =  (id_{\!_{X}} \otimes \theta) (id_{\!_{X}} \otimes p) \alpha $ \\
 & (naturality of $\otimes$)\\
  &  $ = (id_{\!_{X}} \otimes \theta) \beta (p \otimes id_{\!_{X}}) $ \\
  &  (by (\ref{diag.eq44})) \\
\end{tabular}\\

\vspace{.2cm}

Since $Q_X$ is cocontinuous, it preserves epimorphisms. Hence, $Q_X (p) = (p \otimes id_{\!_{X}})$ is an epimorphism. Thus,  $\beta' (\theta \otimes id_{\!_{X}}) = (id_{\!_{X}} \otimes \theta) \beta$.  Therefore, $\theta$ is an isomorphism in $\mathcal{Z}_X(\Cc)$, and the proof is complete.

\end{proof}

\begin{corollary}\label{c.p.cowell7} 
Let $\mathcal{C}$  be a co-wellpowered category. If $\mathcal{Q}_X$ is cocontinuous, $\forall X \in \Cc$, then $\mathcal{Z}(\Cc)$ and $\mathcal{Z}_{\omega}(\Cc)$ are co-wellpowered.
\end{corollary}
\begin{proof}
This immediately follows from the proof of Proposition \ref{p.cowell6} and Remark  \ref{r.Z.X.Z.h}.

\end{proof}

\section{\textbf{Generators}}\label{s.gen}

Following  \cite[p. 127]{Mac Lane1}, a set $\mathcal{G}$  of objects of the category $\mathscr{C}$  is said to \textit{generate} $\mathscr{C}$  when any parallel pair $f,g: X \rightarrow Y$  of arrows of $\mathscr{C}$, $ f \neq  g $  implies that there is an $G \in \mathcal{G}$  and an arrow $\alpha:G \rightarrow X$ in $\mathscr{C}$ with $f\alpha \neq g\alpha$  (the term ``generates" is well established but poorly chosen; ``\textit{separates}" would have been better).  For the basic concepts of generating sets, we refer to   \cite{Mac Lane1},  \cite{Adamek}, or \cite{Freyd}. \\

Let $\Cc$ be a monoidal category with a generating set  $\mathcal{G}$.  Fix an object $X$ and a morphism $h: A \rightarrow B$ in $\mathcal{C}$, and let  $\mathscr{A} \in \{ \mathcal{Z}_h(\Cc), \mathcal{Z}_X(\Cc), \mathcal{Z}(\Cc) ,  \mathcal{Z}_{\omega}(\Cc) \}$. 
Our inspection in the previous sections gives rise to the following question. When can the category  $\mathscr{A}$ inherit a generating set involved with $\mathcal{G}$ from $\Cc$? \\
Under the assumption above, let   $f,g: Z \rightarrow W$ be any parallel pair of morphisms in $\mathscr{A}$ with  $ f \neq  g $. Since   $\mathcal{G}$ is a generating set for $\Cc$, there is an $G \in \mathcal{G}$  and an arrow $\alpha:G \rightarrow X$ in $\Cc$ with $f\alpha \neq g\alpha$. Now, if we want to show that  $\mathscr{A}$ has a  generating set $\mathscr{G}$ whose underlying is $\mathcal{G}$, we need to show that  $G \in \mathscr{A}$  and the morphism $\alpha:G \rightarrow X$ is in $\mathscr{A}$. Although, this is not true in general, it perfectly works when $\Cc$ is a braided. For the basic notions of braided monoidal categories, we refer to \cite{Street},  \cite{Joyal1} and \cite{Majid}.\\

It turns out that we have the following theorem.

\begin{theorem}\label{t.gen.braid} 
Let $\Cc$ be a braided monoidal category with a braiding $\Psi$, and let  $\mathcal{G}$ be a generating set for  $\Cc$.  Fix an object $X$ and a morphism $h$ in $\mathcal{C}$, and let  $\mathscr{A} \in \{ \mathcal{Z}_h(\Cc), \mathcal{Z}_X(\Cc), \mathcal{Z}(\Cc) ,  \mathcal{Z}_{\omega}(\Cc) \}$. The category  $\mathscr{A}$ has a generating set. 
\end{theorem}

\begin{proof} Consider the following diagram 

 \vspace{.2cm}
 
\begin{equation} \label{diag.eq46}
 \xymatrix{
& \mathcal{Z}(\Cc) & \\
\mathcal{Z}_X(\Cc) & \Cc \ar@{^{(}->}[u]^{\Phi_1} \ar@{_{(}->}[d]_{\Phi_4} \ar@{_{(}->}[l]_-{\Phi_2} \ar@{^{(}->}[r]^-{\Phi_3} & \mathcal{Z}_h(\Cc)\\
&  \mathcal{Z}_{\omega}(\Cc) } 
\end{equation}

 \vspace{.2cm}
 
It is well-known that there is an embedding $\Phi_1: \Cc  \hookrightarrow \mathcal{Z}(\Cc)$ via $W \mapsto (W, \Psi_{\!_{{W,-}}})$ \cite[p. 264]{Davydov}. \\

Define the functors\\

\begin{center}
• $\Phi_2: \Cc \hookrightarrow \mathcal{Z}_X(\Cc)$, $W \mapsto (W,  \Psi_{\!_{{W,X}}})$, \\
\end{center}

\vspace{.2cm}

\begin{center}
•  $\Phi_3: \Cc \hookrightarrow \mathcal{Z}_h(\Cc)$, $W \mapsto (W,  \Psi_{\!_{{A,W}}}, \Psi_{\!_{{B,W}}})$, \\
\end{center}

\vspace{.2cm}

\begin{center}
• $\Phi_4: \Cc \hookrightarrow \mathcal{Z}_{\omega}(\Cc)$, $W \mapsto (W, \Psi_{\!_{{W,-}}})$. \\
\end{center}

\vspace{.5cm}

Clearly, $\Phi_i$ is embedding for all $i=2,3,4$. Therefore, the category $\Cc$ can be viewed as a subcategory of the category $\mathscr{A}$, $\mathscr{A} \in \{ \mathcal{Z}_h(\Cc), \mathcal{Z}_X(\Cc), \mathcal{Z}(\Cc) ,  \mathcal{Z}_{\omega}(\Cc) \}$. \\

Now, let   $f,g: Z \rightarrow W$ be any parallel pair of morphisms in $\mathscr{A}$ with  $ f \neq  g $. Since   $\mathcal{G}$ is a generating set for $\Cc$, there is an $G \in \mathcal{G}$  and an arrow $\alpha:G \rightarrow X$ in $\Cc$ with $f\alpha \neq g\alpha$. From the diagram \ref{diag.eq46}, we have $G \in \mathscr{A}$, and the morphism $\alpha:G \rightarrow X$ is in $\mathscr{A}$. Thus, for every  $\mathscr{A} \in \{ \mathcal{Z}_h(\Cc), \mathcal{Z}_X(\Cc), \mathcal{Z}(\Cc) ,  \mathcal{Z}_{\omega}(\Cc) \}$, $\mathscr{A}$ has a  generating set $\mathscr{G}_{\!_{{\mathscr{A}}}}$ whose underlying is $\mathcal{G}$. 

\end{proof}

The following assertion is important in characterizing the cofree objects in $CoMon(\mathscr{A})$, for all  $\mathscr{A} \in \{ \mathcal{Z}_h(\Cc), \mathcal{Z}_X(\Cc), \mathcal{Z}(\Cc) ,  \mathcal{Z}_{\omega}(\Cc) \}$.\\

\begin{corollary}\label{c.gen.braid} 
Let $\Cc$ be a braided monoidal category.  Fix an object $X$ and a morphism $h$ in $\mathcal{C}$ in $\mathcal{C}$, and let  $\mathscr{A} \in \{ \mathcal{Z}_h(\Cc), \mathcal{Z}_X(\Cc), \mathcal{Z}(\Cc) ,  \mathcal{Z}_{\omega}(\Cc) \}$. If the (monoidal) category $CoMon(\Cc)$ has a generating set, then the category  $CoMon(\mathscr{A})$ has a generating set. 
\end{corollary}

\begin{proof}

\end{proof}

 \vspace{.5cm}

\section{\textbf{Investigating Cofree Objects}}\label{s.cofree}

In this section, we use Theorem \ref{p.SAFT} and Propositions \ref{p.cocomp-cowell}, \ref{p.cowell.comon}, and  the consequences we have to show that the concrete category $(CoMon(\mathscr{A}),\mathscr{U}_{\!_{{\mathscr{A}}}})$ has cofree objects, $\forall  \mathscr{A} \in \{ \mathcal{Z}_h(\mathcal{C}), \mathcal{Z}_X(\mathcal{C}),  \mathcal{Z}(\mathcal{C}),  \mathcal{Z}_{\omega}(\mathcal{C})\}$. 

\begin{theorem}\label{t.cofree8} 
Let $h: A \rightarrow B$  be an arrow in $\Cc$. Let $\mathscr{U}: CoMon(\mathcal{Z}_h(\Cc)) \rightarrow \mathcal{Z}_h(\Cc)$ be the forgetful functor and $\mathcal{P}_J,\mathcal{Q}_J$  cocontinuous $\forall J \in \{ A,B\}$. If $\mathcal{C}$ is cocomplete, co-wellpowered and if $CoMon(\mathcal{Z}_h(\Cc)$ has a generating set, then $\mathscr{U}$ has a right adjoint or, equivalently, the concrete category $(CoMon(\mathcal{Z}_h(\Cc)),\mathscr{U})$ has cofree objects. 
\end{theorem}

\begin{proof}
This immediately follows from  Propositions \ref{p.cocomp-cowell}, \ref{p.cocomp1}, \ref{p.cowell5} and Theorem \ref{p.SAFT}. 
\end{proof}

\begin{corollary}\label{c.cofree8.1} 
Let $(\Cc,\otimes,I)$ be a braided monoidal category and $h: A \rightarrow B$  an arrow in $\Cc$. Let $\mathscr{U}: CoMon(\mathcal{Z}_h(\Cc)) \rightarrow \mathcal{Z}_h(\Cc)$ be the forgetful functor and $\mathcal{P}_J,\mathcal{Q}_J$ be cocontinuous $\forall J \in \{ A,B\}$. If $\mathcal{C}$ is cocomplete, co-wellpowered and if $CoMon(\Cc)$ has a generating set, then $\mathscr{U}$ has a right adjoint or equivalently, the concrete category $(CoMon(\mathcal{Z}_h(\Cc)),\mathscr{U})$ has cofree objects. 
\end{corollary}

\begin{proof}
It is immediate from  Theorem \ref{t.cofree8} and Corollary \ref{c.gen.braid}. 
\end{proof}

Similarly, the following are immediate consequences of Propositions \ref{p.cocomp-cowell}, \ref{p.cocomp2}, \ref{p.cowell6} and Theorem \ref{p.SAFT}.

\begin{theorem}\label{t.cofree9}  
Let $X$ be an object in $\Cc$. Let $\mathscr{U}: CoMon(\mathcal{Z}_X(\Cc)) \rightarrow \mathcal{Z}_X(\Cc)$ be the forgetful functor and $\mathcal{P}_X,\mathcal{Q}_X$  cocontinuous.  If $\mathcal{C}$ is cocomplete, co-wellpowered and if $CoMon(\mathcal{Z}_X(\Cc))$ has a generating set, then $\mathscr{U}$ has a right adjoint, hence, the concrete category $(CoMon(\mathcal{Z}_X(\Cc)),\mathscr{U})$ has cofree objects. 
\end{theorem}

\begin{proof}

\end{proof}

\begin{corollary}\label{c.cofree9.1}  
Let $(\Cc,\otimes,I)$ be a braided monoidal category and $X$ an object in $\Cc$. Let $\mathscr{U}: CoMon(\mathcal{Z}_X(\Cc)) \rightarrow \mathcal{Z}_X(\Cc)$ be the forgetful functor and $\mathcal{P}_X,\mathcal{Q}_X$ be cocontinuous.  If $\mathcal{C}$ is cocomplete, co-wellpowered and if $CoMon(\Cc)$ has a generating set, then $\mathscr{U}$ has a right adjoint, hence, the concrete category $(CoMon(\mathcal{Z}_X(\Cc)),\mathscr{U})$ has cofree objects. 
\end{corollary}

\begin{proof}
It follows immediately from  Theorem \ref{t.cofree9} and Corollary \ref{c.gen.braid}. 
\end{proof}

By Propositions \ref{p.cocomp-cowell}, \ref{p.cocomp3}, Corollary \ref{c.p.cowell7} and Theorem \ref{p.SAFT}, we have the following version for the existence of cofree objects in the monoidal center.

\begin{theorem}\label{t.cofree10} 
Let $\mathscr{U}: CoMon(\mathcal{Z}(\Cc)) \rightarrow \mathcal{Z}(\Cc)$ (resp. $\mathscr{U}': CoMon(\mathcal{Z}_{\omega}(\Cc)) \rightarrow \mathcal{Z}(\Cc)$) be the forgetful functor, and  let $\mathcal{P}_X,\mathcal{Q}_X$ be cocontinuous $\forall X \in \Cc$. If $\mathcal{C}$ is cocomplete, co-wellpowered and if $CoMon(\mathcal{Z}(\Cc))$ (resp. $CoMon(\mathcal{Z}_{\omega}(\Cc))$) has a generating set, then $\mathscr{U}$ (resp. $\mathscr{U}'$) has a right adjoint. It turns out that, equivalently, the concrete category $(CoMon(\mathcal{Z}(\Cc)),\mathscr{U})$ (resp. $(CoMon(\mathcal{Z}_{\omega}(\Cc)),\mathscr{U}')$) has cofree objects. 
\end{theorem}

\begin{proof}
\end{proof}

\begin{corollary}\label{c.cofree10.1} 
Let $(\Cc,\otimes,I)$ be a braided monoidal category and $\mathscr{U}: CoMon(\mathcal{Z}(\Cc)) \rightarrow \mathcal{Z}(\Cc)$ (resp. $\mathscr{U}': CoMon(\mathcal{Z}_{\omega}(\Cc)) \rightarrow \mathcal{Z}(\Cc)$) the forgetful functor, and  let $\mathcal{P}_X,\mathcal{Q}_X$ be cocontinuous $\forall X \in \Cc$. If $\mathcal{C}$ is cocomplete, co-wellpowered and if $CoMon(\Cc)$ has a generating set, then $\mathscr{U}$ (resp. $\mathscr{U}'$) has a right adjoint. It turns out that, equivalently, the concrete category $(CoMon(\mathcal{Z}(\Cc)),\mathscr{U})$ (resp. $(CoMon(\mathcal{Z}_{\omega}(\Cc)),\mathscr{U}')$) has cofree objects. 
\end{corollary}

\begin{proof}
The required statement follows from  Theorem \ref{t.cofree10} and Corollary \ref{c.gen.braid}. 
\end{proof}

 \vspace{.2cm}
 
\begin{example} \label{ex.1}
Following  \cite[p. 69-70]{Street}, the braid category $\mathcal{B}$ has as objects the natural numbers
$0, 1, 2, . . .$  and as arrows $\alpha : n \rightarrow n$ the braids on $n$ strings; there are no arrows $n \rightarrow n$ for $m \neq n$. \\

A braid $\alpha$ on $n$ strings can be regarded as an element of the Artin braid group $\mathcal{B}_n$ with generators $s_1, . . . , s_{n-−1} $ subject to the relations
\begin{center}
•$s_is_j = s_js_i$, for $j<i-1$\\
$s_{i+1}s_is_{i+1} = s_i s_{i+1} s_i$
\end{center}

Composition of braids is just multiplication in this group, represented diagrammatically by vertical stacking of braids with the same number of strings. \\

Tensor product of braids adds the number of strings by placing one braid next to the other longitudinally.\\

This makes $\mathcal{B}$  a strict monoidal category. A braiding  
$c_{m,n} : m+n \rightarrow n+m$  is given by crossing the first $m$ strings over the remaining $n$.\\

Then  $\mathcal{B}$  is braided monoidal category. Indeed, it is a balanced monoidal category. To see how  the braid $s_i$, the composition of braids, tensor product of braids and the braiding $ c_{m,n}$ can be depicted,  we refer the reader to \cite[p. 69-70]{Street}.\\
\end{example}

\begin{proposition}\label{p.braid1} 
 The category $\mathcal{B}$  is not cocomplete. 
\end{proposition}

\begin{proof} 
Let $\mathscr{D}$ be a small category, and let $\mathscr{F}:\mathscr{D} \rightarrow \mathcal{B}$ be a functor. By the way of contradiction, let $\mathcal{B}$ be a cocomplete category. It follows that $\mathscr{F}$ has a colimit $(t,(\phi_{\!_{D}})_{\!_{{D} \in \mathscr{D}}})$. The definition of $\mathcal{B}$ implies that $\mathscr{F}D = t$,  for all $D \in \mathscr{D}$. In particular, we have $\mathscr{F}$ is a constant functor, for every functor $\mathscr{F}:\mathscr{D} \rightarrow \mathcal{B}$ with $\mathscr{D}$ a small category. It is clear that this is a contradiction because we can always define a nonconstant functor from a small category to  $\mathcal{B}$. Therefore,  the category $\mathcal{B}$  is not cocomplete.\\

\end{proof}

\begin{theorem}\label{t.braid2} 
Fix an object $X$ and a morphism $h$ in $\mathcal{B}$. We have  \\
\[
\xymatrix{
CoMon(\mathscr{A}) & \cong & \bullet \ar@(ul,ur) & , & \textrm{for all }  \mathscr{A} \in \{ \mathcal{Z}_h(\mathcal{B}), \mathcal{Z}_X(\mathcal{B}), \mathcal{Z}(\mathcal{B}), \mathcal{Z}_{\omega}(\mathcal{B})\}},
\]\\
 \xymatrix{
\textrm{where} & \bullet \ar@(ul,ur) & , & \textrm{is the category with one object and one (identity) arrow.}} 

\end{theorem}

\begin{proof} 
We prove the theorem for $\mathscr{A} = \mathcal{Z}(\mathcal{B})$, and the rest can be proved similarly.  Let $((m,\sigma), \Delta, \epsilon)$ be a comonoid in $\mathcal{Z}(\mathcal{B})$ with a comultiplication $\Delta : m \rightarrow m + m $ and a counit $\epsilon : m \rightarrow 0$. The  definition of the category $\mathcal{B}$ implies that $m = 0$, $\Delta = id_{\!_{0}} = \epsilon$, and $\sigma:0 + - \rightarrow -+0 $ with $ \sigma_{\!_{n}} = id_{\!_{n}}$, for every natural number $n$. Thus, the category $CoMon(\mathcal{Z}(\mathcal{B}))$ consists of one object $((0,\sigma), id_{\!_{0}}, id_{\!_{0}})$, where $\sigma:0 + - \rightarrow -+0 $ is the trivial natural isomorphism with with $ \sigma_{\!_{n}} = id_{\!_{n}}$, for every natural number $n$.

\end{proof}

\begin{theorem}\label{t.braid3} 
Fix an object $X$ and a morphism $h$ in $\mathcal{B}$. The forgetful functor $\mathscr{U}_{\!_{{\mathscr{A}}}}: CoMon(\mathscr{A}) \rightarrow \mathscr{A}$ has a right adjoint $\forall  \mathscr{A} \in \{ \mathcal{Z}_h(\mathcal{B}), \mathcal{Z}_X(\mathcal{B}), \mathcal{Z}(\mathcal{B}),  \mathcal{Z}_{\omega}(\mathcal{B})\}$, and thus, equivalently, the concrete category $(CoMon(\mathscr{A}),\mathscr{U}_{\!_{{\mathscr{A}}}})$ has cofree objects. 
\end{theorem}

\begin{proof}
It follows immediately from Theorem \ref{t.braid2} that  $CoMon(\mathscr{A})$  is cocomplete, co-wellpowered, and with a generating set. Thus, using Theorem \ref{p.SAFT} completes the proof.

\end{proof}

\begin{remark}\label{r.braid4} 
Fix an object $X$ and a morphism $h$ in $\mathcal{B}$. It follows  from Theorem \ref{t.braid3} that the forgetful functor $\mathscr{U}_{\!_{{\mathscr{A}}}}: CoMon(\mathscr{A}) \rightarrow \mathscr{A}$ has a right adjoint $\mathscr{V}_{\!_{{\mathscr{A}}}} :\mathscr{A} \rightarrow  CoMon(\mathscr{A})$ $\forall  \mathscr{A} \in \{ \mathcal{Z}_h(\mathcal{B}), \mathcal{Z}_X(\mathcal{B}), \mathcal{Z}(\mathcal{B}),  \mathcal{Z}_{\omega}(\mathcal{B})\}$.  Theorem \ref{t.braid3}, furthermore, implies that $\mathscr{V}_{\!_{{\mathscr{A}}}} :\mathscr{A} \rightarrow  CoMon(\mathscr{A})$ is a constant functor $\forall  \mathscr{A} \in \{ \mathcal{Z}_h(\mathcal{B}), \mathcal{Z}_X(\mathcal{B}), \mathcal{Z}(\mathcal{B}),  \mathcal{Z}_{\omega}(\mathcal{B})\}$. Therefore, $\forall  \mathscr{A} \in \{ \mathcal{Z}_h(\mathcal{B}), \mathcal{Z}_X(\mathcal{B}), \mathcal{Z}(\mathcal{B}),  \mathcal{Z}_{\omega}(\mathcal{B})\}$, all the  objects in the category $\mathscr{A}$ have the same corresponding cofree object.\\
\end{remark}

\begin{example} \label{ex.2}
Following  \cite[p. 74-75]{Street}, the monoidal category $\tilde{\mathcal{B}}$ is defined similarly to $\mathcal{B}$, except that the arrows are braids on ribbons (instead of on strings)
and it is permissible to twist the ribbons through full $2\pi$ turns. \\

The homsets $\tilde{\mathcal{B}}(n, n) =  \tilde{\mathcal{B}}_n$ are groups under composition. A presentation of this group $\tilde{\mathcal{B}}_n$ is given by generators $s_1, . . . , s_n$ where $s_1, . . . , s_{n-1}$ satisfy the relations as for $\mathcal{B}_n$. These are depicted by thickened versions of the diagrams in Example \ref{ex.1}, along with the extra relation 

\begin{center}
• $s_{n-1} s_n s_{n-1} s_n = s_n s_{n-1} s_n s_{n-1}$
\end{center}

Composition in $\tilde{\mathcal{B}}$  is vertical stacking of diagrams, and tensor product for $\tilde{\mathcal{B}}$ is horizontal placement of diagrams, much as for $\mathcal{B}$. The braiding
$ c_{m,n}: m+n \rightarrow n+m$ for $\tilde{\mathcal{B}}$ is obtained by placing the first $m$ ribbons over the remaining $n$ without introducing any twists.  Then $\tilde{\mathcal{B}}$ is a braided monoidal category. Indeed, it is a balanced monoidal category. To see how $s_n$ and the braiding $ c_{m,n}$ can be visualized,  we refer the reader to \cite[p. 74-75]{Street}.\\

\end{example} 

The identification of $\tilde{\mathcal{B}}$ is similar to that of $\mathcal{B}$, and Proposition \ref{p.braid1} and Theorems \ref{t.braid2}, \ref{t.braid3} imply the following consequences.

\begin{proposition}\label{p.braid.ribbons1}
 Fix an object $X$ and a morphism $h$ in $\tilde{\mathcal{B}}$. For all  $\mathscr{A} \in \{ \mathcal{Z}_h(\tilde{\mathcal{B}}), \mathcal{Z}_X(\tilde{\mathcal{B}}), \mathcal{Z}(\tilde{\mathcal{B}}),  \mathcal{Z}_{\omega}(\tilde{\mathcal{B}})\}$, let $\mathscr{U}_{\!_{{\mathscr{A}}}}: CoMon(\mathscr{A}) \rightarrow \mathscr{A}$  be the forgetful functor. We have the following:
\begin{enumerate}[label=(\roman*)]
\item The category $\tilde{\mathcal{B}}$  is not cocomplete. \\
\item 
\[
\xymatrix{
CoMon(\mathscr{A}) & \cong & \bullet \ar@(ul,ur)}.
\]\\
\item   For all  $\mathscr{A} \in \{ \mathcal{Z}_h(\tilde{\mathcal{B}}), \mathcal{Z}_X(\tilde{\mathcal{B}}), \mathcal{Z}(\tilde{\mathcal{B}}),  \mathcal{Z}_{\omega}(\tilde{\mathcal{B}})\}$, the functor $\mathscr{U}_{\!_{{\mathscr{A}}}}$ has a right adjoint, and thus, equivalently, the concrete category $(CoMon(\mathscr{A}),\mathscr{U}_{\!_{{\mathscr{A}}}})$ has cofree objects. 
 \end{enumerate}

\end{proposition}

 \vspace{.3cm}
 
\begin{center}
• \textbf{Acknowledgment}
\end{center}
I would like to thank my adivsor Prof. Miodrag Iovanov for his support and his suggestions.

\vspace*{3mm} 
\begin{flushright}
\begin{minipage}{148mm}\sc\footnotesize

Adnan Hashim Abdulwahid\\
University of Iowa, \\
Department of Mathematics, MacLean Hall\\
Iowa City, IA, USA

{\tt \begin{tabular}{lllll}
{\it E--mail address} : &   {\color{blue} adnan-al-khafaji@uiowa.edu}\\
&  {\color{blue} adnanalgebra@gmail.com}\\
\end{tabular} }\vspace*{3mm}
\end{minipage}
\end{flushright}

\end{document}